\documentclass[11pt]{amsart}

\usepackage{enumitem}
\usepackage{amssymb}
\usepackage{tikz-cd} 
\usepackage{hyperref}
\hypersetup{colorlinks=true}
\usepackage{dsfont}
\usepackage{mdframed}
\usetikzlibrary{decorations.pathmorphing}
\usepackage[toc,page]{appendix}

\theoremstyle{plain}
\newtheorem{theorem}{Theorem}[section]

\newtheorem{proposition}[theorem]{Proposition}
\newtheorem{lemma}[theorem]{Lemma}
\newtheorem{corollary}[theorem]{Corollary}

\newtheorem*{theorem*}{Theorem}

\theoremstyle{definition}
\newtheorem{definition}[theorem]{Definition}
\newtheorem*{definition*}{Definition}

\theoremstyle{remark}
\newtheorem{remark}{Remark}[section]

\DeclareMathOperator{\Aut}{Aut}

\DeclareMathOperator{\GL}{\mathbf{GL}}

\DeclareMathOperator{\rank}{rank}

\DeclareMathOperator{\id}{id}

\DeclareMathOperator{\Comm}{Comm}
\DeclareMathOperator{\Rad}{Rad}

\DeclareMathOperator{\Isom}{Isom}

\title{Approximate lattices and S-adic linear groups}
\author{Simon Machado \\
ETHZ}
\email{smachado@ethz.ch}
\begin{document}

\begin{abstract}
We provide and motivate in this paper a natural framework for the study of approximate lattices. Namely,  we consider approximate lattices in so-called $S$-adic linear groups and define relevant notions of arithmeticity.  We also adapt to this framework classical results of the theory of lattices and Meyer sets. Results from this paper will play a role in the proof of a structure theorem for approximate lattices in $S$-adic linear groups which is the subject of a companion paper.

 We extend a theorem of Schreiber's concerning the coarse structure of approximate subgroups in Euclidean spaces to approximate subgroups of unipotent $S$-adic groups. We generalise  Meyer's structure theorem for approximate lattices in locally compact abelian groups to a precise structure theorem for approximate lattices in unipotent $S$-adic groups. Finally, we study intersections of approximate lattices of $S$-adic linear groups with certain subgroups such as the nilpotent radical and Levi subgroups. We furthermore show that the framework of $S$-adic linear groups enables us to provide statements more precise than earlier results. 
\end{abstract}

\maketitle

\section{Introduction}

The first instances of approximate lattices where studied in seminal work of Yves Meyer \cite{meyer1972algebraic, moody1997meyer}. There he studied more specifically approximate lattices of Euclidean spaces - now dubbed \emph{Meyer sets} - and their link with Pisot numbers. More recently, Michael Bj\"{o}rklund and Tobias Hartnick launched the study of approximate lattices beyond Euclidean spaces \cite{bjorklund2016approximate}. Following this work, a number of advances in the theory of approximate lattices were made \cite{cordes2020foundations, machado2020apphigherrank, hrushovski2020beyond}.

In this work we study basic properties of approximate lattices in $S$-adic linear groups. A group $G$ is called an $S$-\emph{adic linear group} if there exists a finite set $S$ of inequivalent places of $\mathbb{Q}$ such that $G = \prod_{v \in S} \mathbf{G}(\mathbb{Q}_v)$ where for all $v \in S$, $\mathbb{Q}_v$ denotes the completion of $\mathbb{Q}$ with respect to $v$ and $\mathbf{G}_v$ is a linear algebraic group defined over $\mathbb{Q}_v$. An $S$-adic linear group thus always comes equipped with a Hausdorff topology inherited from the respective topologies of the local fields $\mathbb{Q}_v$.

A subset $\Lambda$ of a group $G$ is an \emph{approximate subgroup} if $e \in \Lambda$, $\Lambda=\Lambda^{-1}$ and there is $F \subset G$ finite such that $\Lambda^2 \subset F \Lambda$. If $G$ is locally compact, $\Lambda$ is moreover an \emph{approximate lattice} if it is discrete and there is a subset $\mathcal{F}$ of finite Haar measure such that $\Lambda \mathcal{F} =G$. Meyer presented in \cite{meyer1972algebraic} a way to build many approximate lattices. A triple $(G,H,\Gamma)$ is called a \emph{cut-and-project scheme} if $G$ and $H$ are locally compact groups and $\Gamma \subset G \times H$ is a lattice. Given a symmetric relatively compact neighbourhood of the identity $W_0 \subset H$ called the \emph{window}, one can build the \emph{model set} 
$$ M:=p_G\left(\Gamma \cap G \times W_0\right).$$
In this construction, $G$ is called the \emph{physical space} of $M$ and $H$ is the \emph{internal space}. All model sets are approximate lattices \cite{bjorklund2016approximate, meyer1972algebraic}. Conversely, Meyer famously showed that all approximate lattices of Euclidean spaces are close to model sets \cite{meyer1972algebraic}.

In what follows we argue that while being a natural and already very general framework presenting many new phenomena, $S$-adic linear groups also provide enough algebraic structure to afford precise statements. The two main results of this paper illustrate this. 

\subsection{Meyer's theorem}
A cornerstone of the theory of approximate lattices in Euclidean spaces - a.k.a.  \emph{Meyer sets} - is the theorem of Meyer asserting that all approximate lattices of Euclidean spaces are commensurable to model sets \cite{meyer1972algebraic}.  Following \cite{bjorklund2016approximate},  there was a recent push for generalisations of Meyer's theorem to more general non-commutative groups. We propose below a result of this nature: 

\begin{proposition}[Meyer's theorem for unipotent $S$-adic linear groups]\label{Proposition: Meyer's theorem $S$-adic unipotent group, intro}
 Let $\Lambda$ be an approximate lattice in a unipotent $S$-adic algebraic group $U$. Then there is a unipotent $\mathbb{Q}$-group $\mathbf{U}$ such that: 
\begin{enumerate}
\item there is a surjective regular group homomorphism $\pi:\mathbf{U}(\mathbb{A}_{S}) \rightarrow U$;
\item $\mathbf{U}(\mathbb{A}_{S})$ is equal to the product $U_1 \times U_2$ of two Zariski-closed unipotent subgroups;
\item $\pi_{|U_1}$ is an isomorphism and $U_2 = \ker \pi$;
\item $\Lambda$ is commensurable with a model set coming from the cut-and-project scheme $(U_1,U_2, \mathbf{U}(\mathbb{Z}_S))$;
\item if $\alpha$ is a continuous automorphism of $U$ that commensurates $\Lambda$, then there is a regular automorphism $\alpha_{\mathbf{U}}$ of $\mathbf{U}(\mathbb{A}_S)$ defined over $\mathbb{Q}$ that stabilises $U_1$ and $U_2$ and such that $\alpha \circ \pi = \pi \circ \alpha_{\mathbf{U}}$. 
\end{enumerate}
\end{proposition}

While Meyer's theorem has already been generalised to approximate lattices of  amenable locally compact groups \cite{machado2019goodmodels} - a class of locally compact groups far richer than unipotent $S$-adic linear groups - the structure of $S$-adic linear groups enables us to obtain a precise description of the internal space associated with $\Lambda$.  Proposition \ref{Proposition: Meyer's theorem $S$-adic unipotent group, intro} also serves to motivate the choice of $S$-adic linear groups in favour of other natural classes of groups such as Lie groups in our considerations. Indeed, by choosing $G$ to be an $S$-adic linear group,  both the physical space ($G$ above) and the internal space ($H$ above) range through the same class of groups. Properties of model sets in $S$-adic linear groups are thus directly related to properties of lattices in $S$-adic linear groups, for which an extensive theory exists.

\subsection{Intersection theorems}
The second family of results we prove here are so-called \emph{intersection theorems}.  Namely,  these are results asserting that the intersection of any approximate lattice with a natural subgroup (e.g.  radical, centre, Levi subgroup) is also an approximate lattice in the subgroup.  These type of results are inspired from a famous theorem of Bieberbach that we recall now: 

\begin{theorem*}[Bieberbach, \cite{Bieberbach}]
Let $m \geq 1$ be an integer and consider $\Isom(\mathbb{R}^m)$. Let $T_m$ denote the normal subgroup formed of all translations of $\mathbb{R}^m$. Then for any lattice $\Gamma \subset \Isom(\mathbb{R}^m)$, the intersection $T_m \cap \Gamma$ is a lattice in $T_m$.  
\end{theorem*}

Our main result of this type concerns approximate lattices of $S$-adic linear groups and is an amalgamation of three intersection results we prove below:

\begin{proposition}[Intersection with the nilpotent radical]\label{Proposition: Unipotent radical is hereditary, intro}
Let $\Lambda$ be a uniform approximate lattice in a solvable $S$-adic linear group $R$. Suppose that $\Lambda$ generates a Zariski-dense subgroup. Let $N$ be the maximal Zariski-closed Zariski-connected nilpotent group: 
\begin{enumerate}
\item $\Lambda^2 \cap \Rad(G)$ is a uniform approximate lattice in the radical $\Rad(G)$;
\item $\Lambda^2 \cap N$ is a uniform approximate lattice in the nilpotent radical $N$;
\item $\Lambda^2 \cap [G,R]$ is a uniform approximate lattice in the unipotent subgroup generated by all commutators of an element of $G$ and an element of $R$.
\end{enumerate}
\end{proposition}

We prove in fact a number of other intersection theorems concerning other types of subgroups, see \S \ref{Section: Intersection theorems for approximate lattices}. We rely on two useful ideas: the use of partial `arithmeticity' theorems - in particular Proposition \ref{Proposition: Meyer's theorem $S$-adic unipotent group, intro}; we also exploit a number of `commutator tricks' inspired from Bieberbach's \cite{Bieberbach}.  A third line of ideas used in \cite{MR2373146, machado2019goodmodels} consists in harnessing amenability properties. We will not use these ideas here but we will build upon results proved in \cite{machado2019goodmodels} using them.

\subsection{Outline of the paper}

In \S \ref{Section: $S$-adic linear groups and Pisot numbers} we present the framework precisely, define $S$-adic linear groups and some of their properties. We also define there Pisot numbers and utilise them to exhibit examples of approximate lattices in $S$-adic linear groups that we call arithmetic. These examples have already been investigated in a number of previous works, see \cite{https://doi.org/10.48550/arxiv.2204.01496, hrushovski2020beyond, machado2020apphigherrank}. 

In \S \ref{Section: Extensions of theorems of Schreiber and Meyer to $S$-adic unipotent groups} we set out to prove Proposition \ref{Proposition: Meyer's theorem $S$-adic unipotent group, intro}. Our first step is to prove a generalisation of a theorem of Schreiber's in $S$-adic unipotent groups concerning the coarse structure of approximate subgroups. We then exploit a proof strategy developed in \cite{machado2019infinite} to extend Meyer's theorem. 

In \S \ref{Section: Intersection theorems for approximate lattices} we prove Proposition \ref{Proposition: Unipotent radical is hereditary, intro} along with a number of variations around the same theme. To that end we also introduce a second - wider - family of approximate lattices of arithmetic origin.  These methods also have consequences regarding the commensurator of an approximate lattice of an $S$-adic linear group. 

Finally, in an appendix (\S \ref{Appendix}) we collect results - some old and some new -  concerning approximate subgroups in a much more general framework. These results are used repeatedly throughout the rest of the paper.  
\subsection{Acknowledgements}

This material is based upon work supported by the National Science Foundation under Grant No. DMS-1926686.
\subsection{Notation}

Given subsets $X$ and $Y$ of $G$ define $XY:=\{xy : x \in X, y \in Y\}$, $X^0:=\{e\}$, $X^1:=X$ and $X^{n+1}=X^nX$ for all $n \geq 0$. Write also $\langle X \rangle$ the subgroup generated by $X$. Note that when $X=X^{-1}$, $\langle X \rangle = \bigcup_{n \geq 0} X^n$. We also define $X^y:=y^{-1}Xy$ and $^{y}X:=yXy^{-1}$. The subsets $X$ and $Y$ are \emph{commensurable} if there exists a finite subset $F \subset G$ such that $X \subset FY \cap YF$ and $Y \subset FX \cap XF$. We define the commensurator $\Comm_G(X)$ of $X$ in $G$ as the subgroup of those $g \in G$ such that $gXg^{-1}$ is commensurable with $X$.

\section{$S$-adic linear groups and Pisot numbers}\label{Section: $S$-adic linear groups and Pisot numbers}

In this section we present an attempt at a unified framework for the study of approximate lattices in $S$-adic linear groups in the spirit of the framework of Margulis' \cite{MR1090825} - although we restrict our attention to groups defined over fields of characteristic $0$. We point out that a similar attempt - with a somewhat different scope - was made in \cite{https://doi.org/10.48550/arxiv.2204.01496} and a large part of what we present here is inspired from it. We furthermore indicate that earlier attempts in the generality of semi-simple $S$-adic groups was made in \cite{hrushovski2020beyond, machado2020apphigherrank}. 

\subsection{$S$-adic linear groups}\label{Subsection: $S$-adic linear groups}

Let $S \subset S_{\mathbb{Q}}$ be a finite subset of inequivalent places of $\mathbb{Q}$. An $S$-adic linear group, or $S$-adic group, is any group $G$ such that there is a family of linear algebraic $(\mathbb{Q}_v)_{v \in S}$-groups $(\mathbf{G}_{v})_{v \in S}$ such that $G = \prod_{v \in S} \mathbf{G}_{v}(\mathbb{Q}_v)$. Given a $\mathbb{Q}$-linear group $\mathbf{G}$, the group of points $\mathbf{G}(\mathbb{A}_S)$ is a typical example of an $S$-adic group. 

If $X \subset G$ is any subset, we define its \emph{Zariski-closure} as the product of the Zariski-closures of the $p_v(X)$ where $p_v: G \rightarrow \mathbf{G}_{v}(\mathbb{Q}_v)$ denotes the natural projection i.e. we equip $G$ with the product topology of the Zariski-topologies arising from each $\mathbf{G}_{v}$. Since the fields $\mathbb{Q}_v$ are local, $G$ also comes equipped with a natural Hausdorff topology. All topological properties (closed, open, connected, etc) are understood in the latter topology unless they are preceded by the prefix 'Zariski'. In addition, $\overline{X}$ denotes the closure in the Hausdorff topology unless stated otherwise. 

If P is a property of algebraic groups, we will say that $G$ has P if and only if all $\mathbf{G}_v$ have P. For instance, we will say that $G$ is semi-simple (potentially with non-trivial centre) if and only if $\mathbf{G}_{v}$ is a semi-simple $k_v$-group for all $v \in S$. As in algebraic groups over fields, Zariski-closed subsets of $G$ satisfy a \emph{descending chain condition} i.e. every descending chain of Zariski-closed subsets eventually stabilises.  We write $C_G(X)$ the centraliser of $X$ in $G$. Note that 
 $$C_G(X)=\prod_{v \in S} C_{\mathbf{G}_{v}(\mathbb{Q}_v)}(p_v(X)).$$ 
Define $\Rad(G)$ as the maximal Zariski-connected soluble normal subgroup of $G$. We call $\Rad(G)$ the \emph{radical} of $G$. It satisfies $\prod_{v \in S} \Rad(\mathbf{G}_{v}(\mathbb{Q}_v))$. When $G$ is semi-simple (i.e. $\Rad(G)=\{e\}$), we define the $S$-rank of $G$ as 
$$ \rank_S(G) :=\sum_{v \in S} \rank_{K_v} (\mathbf{G}_{v}).$$ For a general introduction to such groups see \cite{MR1090825}.

We mention now several well-known facts that we will often use in the sequel. According to the Levi decomposition theorem \cite[VIII Theorem 4.3]{MR620024}, if $G$ is Zariski-connected we have that $G=R \ltimes U$ where $R$ is reductive and $U$ is unipotent. Moreover, there are $S \subset R$ semi-simple and $T \subset R$ a product of tori over the fields $\mathbb{Q}_v$ that centralise one another and such that the map $S \times T \rightarrow R$ has finite kernel (\cite[\S 22.o]{MR3729270}). We will call a \emph{reductive} (resp. \emph{semi-simple}) \emph{Levi subgroup} of $G$ any Zariski-closed subgroup that projects isomorphically to $R$ (resp. $S$). Levi subgroups are characterised as the maximal reductive (resp. semi-simple) subgroups of $G$. Moreover, any two reductive (resp. semi-simple) Levi subgroups are conjugate to one another via an element of $U$. 

For every $v \in S$ take $\mathbf{H}_v \subset \mathbf{G}_v$ algebraic $K_v$-subgroup. If $H$ denotes $\prod_{v \in S} \mathbf{H}_v(K_v)$, then the natural map $G/H \rightarrow \prod_{v \in S} (\mathbf{G}_v/\mathbf{H}_v)(K_v)$ is injective and has finite index image \cite{BorelSerre}. When $G$ and $H$ are assumed unipotent, it becomes a continuous isomorphism. This can be seen through the equivalence between unipotent $S$-adic groups and their Lie algebras that we use a number of times, see \cite[IV.2.4]{zbMATH03670601} for references and \cite[II]{raghunathan1972discrete} for the case of real groups. We refer to \cite{MR3729270, springer2010linear} for background on algebraic groups and to \cite{raghunathan1972discrete} for background on Lie groups and their lattices.
 \subsection{Pisot--Vijayaraghavan--Salem numbers of a number field}\label{Subsubsection: Pisot--Vijayaraghavan--Salem numbers of a number field}
 
 Let $K$ be a number field. Write $S_K$ the set of all equivalence classes of places of $K$. For any $v \in S_K$ let $|\cdot|_v$ denote an absolute value arising from $v$. The completion of $K$ with respect to $|\cdot|_v$ will be denoted $K_v$. Note that the space obtained is independent of the choice of $|\cdot |_v$. When $v$ is non-Archimedean let us denote by $O_v$ the valuation ring of $K_v$. We write $\mathbb{A}_K = \prod_{v \in S_K}'K_v$ the ring of adeles with the usual topology. Here, $\prod'$ denotes the restricted product with respect to the valuation rings $O_v$. If $S$ denotes a subset of $S_K$, then we write moreover $\mathbb{A}_{K,S}=\prod_{v \in S}'K_v$ and $\mathbb{A}_K^S = \prod_{v \notin S}'K_v$. In particular, $\mathbb{A}_K= \mathbb{A}_{K,S} \times \mathbb{A}_K^S$. 

We also define the set $\mathcal{O}_{K,S}$ as the subset of those elements $x$ of $K$ such that $|x|_v \leq 1$ for all $v \notin S$. Again, $\mathcal{O}_{K,S}$ depends only on $S$ and not the choice of absolute values $(| \cdot |_v)_{v \in S}$. When $S$ contains all the Archimedean places of $S_K$, the subset $\mathcal{O}_{K,S}$ is the so-called \emph{ring of $S$-integers}. If, moreover, $S$ is the set of all Archimedean places of $K$, then $\mathcal{O}_{K,S}$ is simply denoted $\mathcal{O}_K$ and is the ring of algebraic integers of $K$. For this and more see for instance \cite{NeukirchAlgebraicNumberTheory}.

Another case of interest is when $S$ consists of a single valuation $v$. Then the subsets $\mathcal{O}_{K,S}$ are the \emph{Pisot--Vijayaraghavan--Salem numbers} of $K \subset K_v$ and they admit a fascinating characterisation in terms of a combination of an additive and a multiplicative conditions. 

\begin{theorem}[Meyer's sum-product phenomenon, \S II.13 \cite{meyer1972algebraic}]\label{Theorem: Meyer's sum-product}
Let $\Lambda$ be a subset of a local field $k$. Suppose that $\Lambda$ is a uniformly discrete approximate subgroup of $k$ (seen as an additive group) which is moreover stable under multiplication i.e. $\Lambda\Lambda \subset \Lambda$. There is a global field $K \subset k$ such that $\Lambda \subset \mathcal{O}_{K,v}$ where $v$ is the valuation inherited from the inclusion $K \subset k$. 
\end{theorem}

Theorem \ref{Theorem: Meyer's sum-product} can be seen as the natural generalisation of the fact that a discrete subring (not necessarily unital) of $\mathbb{R}$ is equal to $n\mathbb{Z}$ for some $n \in \mathbb{Z}$. More generally, when $S$ is arbitrary, $\mathcal{O}_{K,S}$ is stable under product. Also, the diagonal embedding of $\mathcal{O}_{K,S}$ in $\mathbb{A}_{K}^S$ is relatively compact. Since $K$ embedded diagonally in $\mathbb{A}_K$ forms a uniform lattice, the diagonal embedding of $\mathcal{O}_{K,S}$ in $\mathbb{A}_{K,S}$ is uniformly discrete. Moreover, we have that $\mathcal{O}_{K,S}$ is a model set coming from the cut-and-project scheme $(\mathbb{A}_{K,S}, \mathbb{A}_{K}^{S}, K)$ and, hence, a uniform approximate lattice in $\mathbb{A}_{K,S}$ \cite[Prop. 2.13]{bjorklund2016approximate}. In particular, $\mathcal{O}_{K,S}$ is an approximate subgroup stable under products (i.e. an \emph{approximate ring}, see \cite{krupinski2023locally}). This approximate structure is reflected in the following:

\begin{lemma}\label{Lemma: Pisot approximate rings and polynomials}
Let $K, S$ be as above and let $P \in K[X]$. Then:
\begin{enumerate}
\item $P(\mathcal{O}_{K,S})$ is contained in finitely many additive translates of $\mathcal{O}_{K,S}$;
\item if, in addition, $P(0) =0$, then there is $\Lambda$ contained in and commensurable with $\mathcal{O}_{K,S}$ such that $P(\Lambda) \subset \mathcal{O}_{K,S}$. Moreover, $\Lambda$ is a model set associated with $( \mathbb{A}_{K,S}, \mathbb{A}_K^S, K)$.
\end{enumerate} 
\end{lemma}

\begin{proof}
Consider the diagonal embedding $\iota: K \rightarrow \mathbb{A}_K^S$. The set $\mathcal{O}_{K,S}$ is the inverse image by $\iota$ of the subset $\{(x_v)_{v \notin S} : |x_v|_v \leq 1\}$. By Lemma \ref{Lemma: Pull-back of commensurable approximate subgroups}, it thus suffices to show that $\iota\left(P(\mathcal{O}_{K,S})\right)$ is relatively compact in order to prove (1). But $\mathcal{O}_{K,S}$ is a relatively compact subset and $P$ seen as a map from $\mathbb{A}_K^S$ to itself is continuous (recall that $\mathbb{A}_K^S$ is a locally compact $K$-algebra). So, indeed, $\iota\left(P(\mathcal{O}_{K,S})\right)$ is relatively compact. This proves (1). Similarly, (2) is a simple consequence of the continuity of $P$ over the locally compact algebra $\mathbb{A}_K^S$.
\end{proof}

When $K = \mathbb{Q}$, we further simplify the notation and write $\mathbb{A}_{S}=\mathbb{A}_{K,S}$, $\mathbb{A}^{S}=\mathbb{A}_{K}^{S}$ and $\mathcal{O}_{K,S}=\mathbb{Z}_S$.
\subsection{Matrices with Pisot entries, Borel--Harish-Chandra and Godement's criterion}\label{Subsection: Matrices with Pisot entries and the Borel--Harish-Chandra theorem for approximate lattices}
The subsets $\mathcal{O}_{K,S}$ introduced in \S \ref{Subsubsection: Pisot--Vijayaraghavan--Salem numbers of a number field} allow us to build a rich family of uniformly discrete approximate subgroups of $S$-adic algebraic groups that happen to be model sets in a great number of situations. We follow here the same approach as \cite{https://doi.org/10.48550/arxiv.2204.01496}.

\begin{definition}[$\mathcal{O}_{K,S}$ points of a linear algebraic group]
Let $K$ be a number field and $S$ a set of inequivalent places. If $\mathbf{G} \subset \GL_n$ is a $K$-subgroup, then define
$$\mathbf{G}(\mathcal{O}_{K,S}):=\{g \in \mathbf{G}(K) : g - \id ,g^{-1} - \id \in \mathbf{M}_{n\times n}\left(\mathcal{O}_{K,S}\right) \},$$
where $\mathbf{M}_{n\times n}\left(\mathcal{O}_{K,S}\right)$ denotes the set of $n \times n$ matrices with entries in $\mathcal{O}_{K,S}$. 
\end{definition}

 When $S=S_1 \sqcup S_2$ recall that $$\mathbf{G}(\mathbb{A}_{K,S})= \mathbf{G}(\mathbb{A}_{K,S_1}) \times \mathbf{G}(\mathbb{A}_{K,S_2}).$$ 
 If $S_{\infty}$ denotes the set of infinite places of $K$, then the diagonal embedding 
 $$\mathbf{G}(\mathcal{O}_{K,S \cup S_{\infty}}) \subset \mathbf{G}(\mathbb{A}_{K,S \cup S_{\infty}})$$
makes $\mathbf{G}(\mathcal{O}_{K,S \cup S_{\infty}})$ into a discrete subgroup. When $\mathbf{G}$ is moreover known to have no $K$-characters, $\mathbf{G}(\mathcal{O}_{K,S \cup S_{\infty}})$ is a lattice in $\mathbf{G}(\mathbb{A}_{K,S \cup S^{\infty}})$ by the Borel--Harish-Chandra theorem \cite{BorelHarish-Chandra}. By applying the general cut-and-project construction and noticing that $\Lambda:=\mathbf{G}(\mathcal{O}_{K,S})$ is a model set associated with the cut-and-project scheme $(\mathbf{G}(\mathbb{A}_{K,S}), \mathbf{G}(\mathbb{A}_{K,S_{\infty}\setminus S}), \mathbf{G}(\mathcal{O}_{K,S \cup S_{\infty}}))$ we have: 

\begin{proposition}
Let $K$ be a number field and $S$ be a set of inequivalent places. Let $\mathbf{G} \subset \GL_n$ be a $K$-subgroup. Then $\mathbf{G}(\mathcal{O}_{K,S}) \subset \mathbf{G}(\mathbb{A}_{K,S})$ is an approximate lattice if and only if $\mathbf{G}$ has no non-trivial $K$-characters. Moreover, if $\mathbf{G}$ is $K$-anisotropic, then $\mathbf{G}(\mathcal{O}_{K,S})$ is uniform. 
\end{proposition}

\begin{proof}
According to the paragraph preceding the statement, if $\mathbf{G}$ has no $K$-characters, we already have that $\mathbf{G}(\mathcal{O}_{K,S})$ is a model set,  hence an approximate lattice by \cite{bjorklund2016aperiodic}. If $\mathbf{G}$ is moreover $K$-anisotropic, then $\mathbf{G}(\mathcal{O}_{K,S \cup S_{\infty}})$ is in fact a uniform lattice in $\mathbf{G}(\mathbb{A}_{K,S \cup S^{\infty}})$. So $\mathbf{G}(\mathcal{O}_{K,S})$ is a uniform approximate lattice, see e.g. \cite[Prop. 2.13]{bjorklund2016approximate}.

It remains to prove the converse statement. Suppose that $\mathbf{G}(\mathcal{O}_{K,S})$ is an approximate lattice. We will see (Lemma \ref{Lemma: Levi decomposition of GAAS}) that it suffices to show that $\mathbf{G}_m(\mathcal{O}_{K,S})$ is never an approximate lattice in $\mathbf{G}_m(\mathbb{A}_{K,S})$ where $\mathbf{G}_m$ stands for the multiplicative group. But if $x \in \mathbf{G}_m(\mathcal{O}_{K,S})$, then the product formula implies $\prod_{v \in S_K}|x|_v = 1$ where $S_K$ denotes the set of all the places of $K$. But $\prod_{v \notin S}|x|_v \leq 1$ by assumption. So $\prod_{v \in S}|x|_v \geq 1$. Since $x^{-1} \in \mathbf{G}_m(\mathcal{O}_{K,S})$ as well, $\prod_{v \in S}|x|_v = 1$. So $\mathbf{G}_m(\mathcal{O}_{K,S})$ does not have finite co-volume.
\end{proof}

If $S'$ denotes $S_{\infty}\setminus S$, then the map $\mathbf{G}(\mathcal{O}_{K,S \cup S'}) \rightarrow \mathbf{G}(\mathbb{A}_{K,S'})$ is a good model of $\mathbf{G}(\mathcal{O}_{K,S})$ and $\mathbf{G}(\mathbb{A}_{K,S'})$ has finitely many connected components. By Proposition \ref{Proposition: Minimal commensurable approximate subgroup}, for any approximate subgroup $\Lambda$ commensurable with $\mathbf{G}(\mathcal{O}_{K,S})$ contained in $\langle \mathbf{G}(\mathcal{O}_{K,S}) \rangle$, we have that $\langle \Lambda \rangle$ is commensurable with $\mathbf{G}(\mathcal{O}_{K,S \cup S'})$.

\section{Extensions of theorems of Schreiber and Meyer to $S$-adic nilpotent groups}\label{Section: Extensions of theorems of Schreiber and Meyer to $S$-adic unipotent groups}

This section is concerned with extending two cornerstone results of approximate subgroups of Euclidean spaces, namely Schreiber's theorem (see \cite{fish2019extensions, schreiber1973approximations}) and Meyer's theorem (see \cite{meyer1972algebraic}). 
\subsection{An extension of Schreiber's theorem}
Schreiber's result asserts that an approximate subgroup of a Euclidean space is \emph{coarsely} the same as a vector subspace \cite{schreiber1973approximations}. We start this subsection by showing a generalisation of Schreiber's theorem in unipotent $S$-adic groups.

\begin{proposition}\label{Proposition: Schreiber's theorem for S-adic unipotent groups}
Let $\Lambda$ be an approximate subgroup of a unipotent $S$-adic linear group $U$. Then there are a unique unipotent $S$-adic linear subgroup $U' \subset U$ and a compact subset $K \subset U$ such that $\Lambda \subset KU'$ and $U' \subset K\Lambda.$
\end{proposition}

\begin{proof}
The structure of the proof roughly goes as follows: the result is known if $U$ is connected (i.e. $S=\{\infty\}$) by \cite{machado2019infinite} and if $U$ is totally disconnected by \cite[Appendix I]{MR3092475}.  To reduce the general case to these two cases, we can exploit the fact that the map $x \mapsto x^p$ is contracting in a $p$-adic unipotent group to show that most of the complexity reduces to understanding the intersection $\Lambda$ with a small neighbourhood around the connected component of $U$.  Although not entirely necessary, and because we will recycle some of these arguments later, we in fact start by operating a reduction to the case where $\Lambda$ is a subgroup (rather than approximate subgroups). We will use the theory of good models to reduce the question to this case.  This enables us to streamline the strategy described at the start of this paragraph. This is carried out in the next paragraph.

Assume thus that $\Lambda$ is equal to the group it generates $\Gamma$. Upon considering the Zariski-closure of $\Gamma$ we may suppose moreover that $\Gamma$ is Zariski-dense in $U$. Write $U=U_{\infty} \times U_{p_1} \times \ldots \times U_{p_n}$ where $U_{\infty}$ is a unipotent $\mathbb{R}$-group and $U_{p_i}$ is a unipotent $\mathbb{Q}_{p_i}$-group for all $i \in \{1, \ldots, n\}$. Let $U_{fin}$ denote $U_{p_1} \times \ldots \times U_{p_n}$. Choose a compact open subgroup $O \subset U_{fin}$. Then the projection of $\Gamma \cap \left(U_{\infty} \times O\right)$ to $U_{\infty}$ is a subgroup and must be contained and co-compact in a closed unipotent subgroup $U'_{\infty} \subset U_{\infty}$ (\cite[\S 2]{raghunathan1972discrete}). Furthermore, the map $g \mapsto g^{p_1\cdots p_n}$ is contracting in $U_{fin}$. So for all $ \gamma \in \Lambda$, there is an integer $m \geq 0$ such that $\gamma^{(p_1 \cdots p_n)^m} \in U_{\infty} \times O$. So the projection of $\gamma^{(p_1 \cdots p_n)^m}$ to $U_{\infty}$ is contained in $U'_{\infty}$. Hence, the projection of $\gamma$ to $U_{\infty}$ is contained in $U_{\infty}'$ i.e. $U_{\infty}'=U_{\infty}$ since $\Gamma$ is Zariski-dense.  Let $p_{fin}: U_{\infty} \times U_{fin} \rightarrow U_{fin}$ denote the natural projection. Then we can invoke \cite[Appendix I]{MR3092475} and find $U_{fin}' \subset U_{fin}$ normal and a compact subset $K_{fin} \subset U_{fin}$ such that $p_{fin}(\Gamma) \subset K_{fin}U_{fin}'$ and  $U_{fin}' \subset K_{fin}p_{fin}(\Gamma)$. Therefore, 
$$ \Gamma \subset U_{\infty} \times K_{fin}U_{fin}'.$$
Also, 
\begin{align*}
U_{\infty} \times U_{fin}' &\subset K_{fin}\Gamma U_{\infty} \\
                           &=  \left( K_{fin} \times U_{\infty}\right)\Gamma  \\
                           & \subset K_{fin}K_{\infty}\left(\Gamma \cap U_{\infty} \times O\right)\Gamma \\
                           & = K_{fin}K_{\infty}\Gamma
\end{align*}
for some compact subset $K_{\infty}$ such that $U_{\infty} \times O = K_{\infty}\left(\Gamma \cap \left(U_{\infty} \times O\right)\right)$.  Such a compact subset exists as the projection of $\Gamma \cap \left(U_{\infty} \times O\right)$ to $U_{\infty}$ was shown to be co-compact in $U_{\infty}$ and $O$ is compact. 

Let us come back to the case of approximate subgroups. By Proposition \ref{Proposition: Minimal commensurable approximate subgroup}, there is an approximate subgroup $\Lambda' \subset \langle \Lambda \rangle$ commensurable with $\Lambda$ such that $\Lambda'$ has a good model $f: \langle \Lambda' \rangle \rightarrow H$ with $H$ a connected nilpotent Lie group without normal compact subgroup i.e. a real unipotent group \cite{raghunathan1972discrete}. Let  $\Gamma_f:=\{(\gamma, f(\gamma)) : \gamma \in \langle \Lambda' \rangle\} \subset U \times H$ be the graph of $f$. By the first part of the proof, we have $U' \subset U \times H$ unipotent and a compact subset $K \subset U \times H$  such that $\Gamma_f \subset KU'$ and $U' \subset K \Gamma_f$. Consider now $\Lambda_f:=\{(\lambda, f(\lambda)): \lambda \in \Lambda'\} \subset \Gamma_f$. The projection of $U' \cap K^{-1} \Lambda_f$ to $H$ is relatively compact. Since the projection of $U'$ to $H$ is closed (e.g. \cite{raghunathan1972discrete}), this means that there is a compact subset $K'$ such that $\Lambda_f \subset K'(U' \cap U)$. Conversely, $\Gamma_f \cap K^{-1}(U' \cap U)$ also projects to a relatively compact subset $K''$ of $H$.  So any $(\gamma, f(\gamma)) \in \Gamma_f \cap K^{-1}(U' \cap U)$ satisfies $\gamma \in f^{-1}(K'')$. As $f$ is a good model of $\Lambda$,  this yields that $\Gamma_f \cap K^{-1}(U' \cap U)$ must be covered by finitely many (right-)translates of $\Lambda_f$. Since $U' \cap U \subset K\left(\Gamma_f \cap K^{-1}(U' \cap U)\right)$, there is a compact subset $K'''$ such that $U' \cap U \subset K'''\Lambda_f$.  Now if $K_0$ denotes the projection of $K'''$ to $U$, we have $\Lambda \subset K_0(U' \cap U)$ and $U' \cap U \subset K_0 \Lambda $ - here we have implicitly used the fact that $\Lambda_f$ projects to $\Lambda$. Hence, $U' \cap U$ is as desired. 

Finally, if $U'$ and $U''$ are two such groups, then there is a compact subset $K$ such that $U'' \subset KU'$ and $U' \subset KU''$. This implies $U'=U''$ since both are unipotent subgroups of $U$, see \cite[\S 2]{raghunathan1972discrete} and \cite[Appendix I]{MR3092475}. 
\end{proof}

The uniqueness of $U'$ implies: 

\begin{corollary}
Let $\Lambda$, $U$ and $U'$ be as above. If an automorphism $\alpha$ of $U$ commensurates $\Lambda$, then $\alpha(U') = U'$. In particular, $U'$ is normalised by $\Comm_U(\Lambda)$. 
\end{corollary}

\subsection{Meyer's theorem for unipotent $S$-adic linear groups}\label{Subsection: Meyer's theorem and its generalisations}

We can now extend Meyer's theorem to $S$-adic unipotent groups following a strategy from \cite{machado2019infinite}.  An extension of Meyer's theorem is known much more generally in amenable groups by \cite{machado2019goodmodels} (and we use that fact to simplify the proof below). Our goal however is really to gain control over the good model created. 

\begin{proposition}\label{Proposition: Meyer's theorem $S$-adic unipotent group}
 Let $\Lambda$ be an approximate lattice in a unipotent $S$-adic algebraic group $U$. Then there is a unipotent $\mathbb{Q}$-group $\mathbf{U}$ such that: 
\begin{enumerate}
\item there is a surjective regular group homomorphism $\pi:\mathbf{U}(\mathbb{A}_{S}) \rightarrow U$;
\item $\mathbf{U}(\mathbb{A}_{S})$ is equal to the product $U_1 \times U_2$ of two Zariski-closed unipotent subgroups;
\item $\pi_{|U_1}$ is an isomorphism and $U_2 = \ker \pi$;
\item $\Lambda$ is commensurable with a model set coming from the cut-and-project scheme $(U_1,U_2, \mathbf{U}(\mathbb{Z}_S))$;
\item if $\alpha$ is a continuous automorphism of $U$ that commensurates $\Lambda$, then there is a regular automorphism $\alpha_{\mathbf{U}}$ of $\mathbf{U}(\mathbb{A}_S)$ defined over $\mathbb{Q}$ that stabilises $U_1$ and $U_2$ and such that $\alpha \circ \pi = \pi \circ \alpha_{\mathbf{U}}$. 
\end{enumerate}
\end{proposition}

\begin{proof}
According to \cite{machado2019goodmodels} we know that $\Lambda$ is laminar. By Proposition \ref{Proposition: Minimal commensurable approximate subgroup}, there is an approximate lattice $\Lambda' \subset \langle\Lambda\rangle$ commensurable with $\Lambda$ that has a good model $f: \langle \Lambda' \rangle \rightarrow H$ with dense image where $H$ is a connected nilpotent Lie group without compact normal subgroup i.e. a real unipotent group. Let $\Gamma_f \subset U \times H$ denote the graph of $f$. Then $\Gamma_f$ is a uniform lattice in $U \times H$ (Proposition \ref{Proposition: Equivalence good models and model sets}). In particular, $\Gamma_f$ is Zariski-dense. Let $U_{\infty}$ denote the factor of $U$ defined over $\mathbb{R}$ and let $U_{fin}$ denote the totally disconnected factor. Write $p_{\infty}: U \times H \rightarrow U_{\infty} \times H$. The kernel of $p_{\infty}$ is $U_{fin}$ and, thus, does not contain any non-trivial discrete subgroup. So $\Gamma_f$ projects injectively to $U_{\infty}$. Choose $O$ a compact open subgroup of $U_{fin}$ and let $\Delta$ denote $p_{\infty}\left(\Gamma_f \cap \left(U_{\infty} \times O  \times H\right)\right)$. Then $\Delta$ is a lattice in $U_{\infty} \times H$ (for, notice that $\Delta$ is a model set, hence and approximate lattice, and a subgroup). So there is a unipotent $\mathbb{Q}$-group $\mathbf{U}$ such that we can identify $\mathbf{U}(\mathbb{R})$ with $U_{\infty} \times H$ and $\mathbf{U}(\mathbb{Z})$ with $\Delta$ \cite{raghunathan1972discrete} (upon, possibly, modifying the compact subgroup $O$). 

Let $p_1, \ldots, p_n$ be the prime numbers appearing in $S$. The projection of $\Gamma_f$ to $U_{fin}$ is Zariski-dense and relatively dense. So it is dense for the Hausdorff topology since $U_{fin}$ is a totally disconnected unipotent $S$-adic group.  The projection of $\Gamma_f \cap \left(U_{\infty} \times O  \times H\right)$ to $O$ is therefore dense in $O$. By \cite{zbMATH03695461}, there is therefore a continuous group homomorphism $\prod_{i=1}^n\mathbf{U}(\mathbb{Z}_{p_i}) \rightarrow U_{fin}$, which extends to a unique regular group homomorphism $\pi_1:\mathbf{U}(\mathbb{A}_{S \setminus \{\infty\}}) \rightarrow U_{fin}$. Let $\pi_2: \mathbf{U}(\mathbb{A}_{S}) \rightarrow U \times H$ denote the product of $\pi_1$ and the isomorphism $\mathbf{U}(\mathbb{R})\rightarrow U_{\infty} \times H$. Remark that $\pi_2$ is injective on $\mathbf{U}(\mathbb{Q})$ and $\pi_2(\mathbf{U}(\mathbb{Z})) =  \Gamma_f \cap \left(U_{\infty} \times O  \times H\right)$.

Now, for all $\gamma \in \Gamma_f$ there is an integer $m \geq 1$ such that $$\gamma^{\left(p_1 \cdots p_n\right)^m} \in \left(U_{\infty} \times O \times H\right) \cap \Gamma_f,$$
see the proof of Proposition \ref{Proposition: Schreiber's theorem for S-adic unipotent groups} for details. But $\left(U_{\infty} \times O \times H\right) \cap \Gamma_f = \pi_2(\mathbf{U}(\mathbb{Z}))$. In other words, $\Gamma_f \subset \pi_2\left(\mathbf{U}(\mathbb{Z}_S)\right)$. Let $\Gamma \subset \mathbf{U}(\mathbb{Z}_S)$ denote the pull-back of $\Gamma_f$ through $\pi_2$ restricted to $\mathbf{U}(\mathbb{Q})$. 

Let $\pi$ denote the composition of $\pi_2$ with the natural projection $U \times H \rightarrow U$. Then the restriction of $\pi$ to $\mathbf{U}(\mathbb{Q})$ is injective. Denote by $\Lambda''$ the set $\mathbf{U}(\mathbb{Q}) \cap \pi^{-1}(\Lambda')$. We know that $\Lambda'' \subset \mathbf{U}(\mathbb{Z}_S)$ and is commensurated by $\Gamma$ which is Zariski-dense in $\mathbf{U}$.  By Proposition \ref{Proposition: Schreiber's theorem for S-adic unipotent groups} applied to $\Lambda''$ we find a Zariski-closed normal subgroup $U_1$ of $\mathbf{U}(\mathbb{A}_S)$ and a compact subset $K$ such that $\Lambda'' \subset KU_1$ and $U_1 \subset K\Lambda''$. We claim that $U_1$, $\pi$ and $U_2:=\ker \pi$ are as required. Indeed, since $\Lambda'$ is relatively dense in $U$ and $\pi(KU_1)$ contains $\Lambda'$, $\pi(U_1)$ is relatively dense in $U$. Hence, $\pi(U_1)=U$ as it is Zariski-closed. Moreover, since $\Lambda'$ is uniformly discrete, $\Lambda' \cap \pi(K^{-1})$ is finite. So $\Lambda'' \cap K^{-1}\left(U_1 \cap U_2\right)$ is finite because the restriction of $\pi$ to $\Lambda''$ is injective. As $U_1 \cap U_2 \subset K \Lambda''$, we have
$$U_1 \cap U_2  \subset K\left( \Lambda'' \cap K^{-1}\left(U_1 \cap U_2\right)\right).$$
So $U_1 \cap U_2$ is compact and, hence, trivial. We have indeed $\mathbf{U}(\mathbb{A}_S) = U_1 \times U_2$. We have that (1)-(3) are satisfied for these choices of $U_1, U_2$ and $\pi$. 

Let us prove (4). For simplicity, identify $U$ and $U_1$ and $\pi$ with the natural projection to $U_1$. Since $\Lambda'' \subset KU_1$, the projection of $\Lambda''$ to $U_2$ is relatively compact. Since $\Lambda''$ is contained in the lattice $\mathbf{U}(\mathbb{Z}_S)$, $\pi(\Lambda'')$ is therefore contained in a model set coming from the cut-and-project scheme $(U_1,U_2,\mathbf{U}(\mathbb{Z}_S))$.  But $\pi(\Lambda'')=\Lambda' $ which is a uniform approximate lattice.  By Lemma \ref{Corollary: Towers of star-approximate subsets are commensurable}, $\Lambda'$ is commensurable with said model set. 

It remains to prove (5). Let $\alpha$ be as in the statement of (5). Then $\alpha$ commensurates $\Lambda'$ and $\langle \Lambda' \rangle$ (Proposition \ref{Proposition: Minimal commensurable approximate subgroup}). Moreover, there is a regular automorphism $\alpha_H$ of $H$ such that $f \circ \alpha = \alpha_H \circ f$ on a finite index subgroup of $\langle \Lambda' \rangle$.  In other words, the product automorphism $\alpha \times \alpha_H$ commensurates $\Gamma_f$. In particular, $(\alpha \times \alpha_H)_{|U_{\infty} \times H}$ commensurates $\mathbf{U}(\mathbb{Z})$ (identified with the projection of $\Gamma_f \cap \left(U_{\infty} \times O  \times H\right)$ to $U_{\infty} \times H$). So there is a $\mathbb{Q}$-automorphism $\alpha_{\mathbf{U}}$ of $\mathbf{U}$ such $\alpha_{\mathbf{U}}: \mathbf{U}(\mathbb{R}) \rightarrow \mathbf{U}(\mathbb{R})$ is equal to $(\alpha \times \alpha_H)_{|U_{\infty} \times H}$, see e.g. \cite[Thm 2.11]{raghunathan1972discrete}.  By construction, $(\alpha \circ \pi)_{| \langle \Lambda' \rangle} = (\pi \circ \alpha_{\mathbf{U}})_{| \langle \Lambda' \rangle}$. So $\alpha \circ \pi= \pi \circ \alpha_{\mathbf{U}}$ as $\langle \Lambda' \rangle$ is Zariski-dense in $\mathbf{U}$.  So $\alpha_{\mathbf{U}}(U_1) = U_1$ and $\alpha_{\mathbf{U}}(U_2) = U_2$, since $\pi$ is the projection to $U_1$ parallel to $U_2$.  
\end{proof}

The ideas behind the proof of (5) can also show that the good model constructed is fairly canonical. 

\begin{corollary}\label{Remark: Meyer theorem is canonical}
Let $\Lambda$ be an approximate lattice in an $S$-adic unipotent group $U$.  Let $\phi: U \rightarrow U'$ be a surjective group homomorphism with target an $S$-adic unipotent group.  Suppose that $\phi(\Lambda)$ is an approximate lattice in $U'$ and let $\mathbf{U}$ and $\mathbf{U}'$  be given by Proposition \ref{Proposition: Meyer's theorem $S$-adic unipotent group} applied to $\Lambda$ and $\phi(\Lambda)$ respectively.  Then there is a homomorphism of $\mathbf{Q}$-groups $\psi: \mathbf{U} \rightarrow \mathbf{U}'$ such that the triple $(\psi(U_1),\psi(U_2), \mathbf{U}'(\mathbb{Z}_S))$ is a  cut-and-project scheme as in the conclusion of Proposition \ref{Proposition: Meyer's theorem $S$-adic unipotent group}.
\end{corollary}

\begin{proof}
A good model $f'$ of the image of $\Lambda$ in $U'$ is simply given by the composition of $f$ and the projection of $H \rightarrow H/N$ where $N=\overline{f\left( \ker \phi \cap \langle \Lambda \rangle\right)}$ \cite[Lem.  3.3]{machado2019goodmodels}.  Now, $H/N$ is also a connected nilpotent Lie group (hence, a real unipotent group) and $\Gamma_{f'}$ is a factor of $\Gamma_f$. Following the constructions of $\mathbf{U}'$ and $\mathbf{U}$ in parallel, we reach the desired conclusion.
\end{proof}

\section{Intersection theorems for approximate lattices}\label{Section: Intersection theorems for approximate lattices}
This section is concerned with understanding how approximate lattices behave around certain natural subgroups of their ambient group. The archetypical theorems are similar to standard results due to Bieberbach and Mostow.

Mostow's result is concerned with a much larger class of groups. 

\begin{theorem*}[Mostow, \cite{MR0289713}]
Let $\Gamma$ be a lattice in a connected Lie group $G$ that has no compact normal semi-simple subgroup. Let $N$ denote the nilpotent radical of $G$ (i.e. the maximal connected nilpotent normal subgroup of $G$). Then $\Gamma \cap N$ is a lattice in $N$. 
\end{theorem*}

By \emph{intersection theorem}, we understand a theorem that asserts that every approximate lattice (or, lattice) of a given group intersects a given natural subgroup in an approximate lattice. The natural subgroups in question are sometimes called \emph{hereditary}. These theorems tell us that, to a certain extent, approximate lattices \emph{comply} with the structure of their ambient group. 

Intersection theorems of lattices have a long history (and a long one of errors), we refer to the survey \cite{MR3358545}. The standard techniques employed are essentially based on ingenious manipulations of commutators, and we largely follow this strategy below. The easiest consequence of this point of view is the fact that centralisers of elements in the commensurator of a uniform approximate lattice are hereditary, see \S \ref{Subsubsection: Centralisers} below. More recently, a strategy involving a strengthening of the Tits alternative was discovered by Breuillard and Gelander \cite{MR2373146}. Following this philosophy, we obtained in \cite{machado2019goodmodels} the first intersection theorem for approximate lattices, see \S \ref{Subsubsection: Solvable radical} below.

\subsection{A survey of intersection theorems and elementary considerations}
\subsubsection{Intersections and projections}
A useful perk of intersection results is that they also inform us on projections.

\begin{proposition}[Prop. 6.2, \cite{machado2019goodmodels}]\label{Proposition: Intersection and projections approximate lattices w/ closed subgroups}
Let $\Lambda$ be an approximate lattice in a locally compact group $G$. Let $N$ be a closed normal subgroup and let $p:G \rightarrow G/N$ denote the natural projection. Then the following are equivalent: 
\begin{enumerate}
\item $p(\Lambda)$ is an approximate lattice in $G/N$;
\item $p(\Lambda)$ is uniformly discrete;
\item $\Lambda^2 \cap N$ is an approximate lattice in $N$. 
\end{enumerate}
If, moreover, $\Lambda$ is uniform, then (2) and (3) are equivalent for all $N$ closed (not necessarily normal). 
\end{proposition}

Proposition \ref{Proposition: Intersection and projections approximate lattices w/ closed subgroups} admits a useful converse. 

\begin{proposition}\label{Proposition: Converse to Intersection and projections approximate lattices w/ closed subgroups}
Let $\Lambda$ be an approximate \emph{subgroup} in a locally compact group $G$. Let $N$ be a closed normal subgroup and let $p:G \rightarrow G/N$ denote the natural projection. If both $\Lambda^2 \cap N$ and $p(\Lambda)$ are approximate lattices in $N$ and $G/N$ respectively, then $\Lambda$ is an approximate lattice. 
\end{proposition}

\begin{proof}
Since $\Lambda$ is already an approximate subgroup, it suffices to show that it is uniformly discrete and has finite co-volume. Let $(\lambda_n)_{n \geq 0}$ be a sequence of elements of $\Lambda^2$ converging to $e$. Then $p(\lambda_n)$ converges to $e$ in $p(\Lambda^2)$. But $p(\Lambda^2)$ is discrete. So $\lambda_n \in N$ for $n$ sufficiently large. As $\Lambda^2 \cap N$ is uniformly discrete as well, $\lambda_n = e$ for $n$ sufficiently large. So $\Lambda$ is uniformly discrete. 

Now $p(\Lambda)$ has finite co-volume so there is $\mathcal{F}_{G/N}$ with finite Haar measure such that $p(\Lambda) \mathcal{F}_{G/N} = G/N$. Similarly, there is $\mathcal{F}_{N}$ with finite Haar measure such that $\left(\Lambda^2 \cap N\right) \mathcal{F}_{N} = N$. But, then $\Lambda^3 \mathcal{F}_N \mathcal{F}_{G/N} =G$ - where $\mathcal{F}_{G/N}$ is identified with a Borel section in $G$ - and $\mathcal{F}_N \mathcal{F}_{G/N}$ has finite Haar measure, see e.g.  \cite[\S I]{raghunathan1972discrete}. 
\end{proof}

\begin{remark}
The proof of Proposition \ref{Proposition: Converse to Intersection and projections approximate lattices w/ closed subgroups} illustrates well the methods involved in the proof of Proposition \ref{Proposition: Intersection and projections approximate lattices w/ closed subgroups}.
\end{remark}

Proposition \ref{Proposition: Intersection and projections approximate lattices w/ closed subgroups} implies that under a relative density assumption, approximate lattices form a family closed under intersection.

\begin{corollary}[Stability under intersection]\label{Lemma: Stability approximate lattices under intersection ambient group}
Let $\Lambda$ be an approximate lattice in a locally compact group $G$. Let $H_1, H_2$ be closed subgroups such that $\Lambda^2 \cap H_i$ is a uniform approximate lattice in $H_i$ for $i=1,2$. Then $\Lambda^2 \cap H_1 \cap H_2$ is a uniform approximate lattice. 
\end{corollary}

\subsubsection{Centralisers}\label{Subsubsection: Centralisers}
Uniform approximate lattices intersect centralisers into approximate lattices. This can be extended to intersections with fixators of commensurating automorphisms.

\begin{lemma}\label{Lemma: Intersection with centraliser}
Let $\Lambda$ be a uniform approximate lattice in a locally compact group $G$. Let $\alpha \in \Aut(G)$ be any element commensurating $\Lambda$ (i.e. $\alpha(\Lambda)$ and $\Lambda$ are commensurable). Write $E(\alpha)$ the closed subgroup defined by $\{g \in G: \alpha(g)=g\}$. Then $E(\alpha) \cap \Lambda^2$ is a uniform approximate lattice in $E(\alpha)$. 
\end{lemma}

\begin{proof}
Since $\alpha(\Lambda)$ is commensurable with $\Lambda$, the subset $\{\alpha(\lambda)\lambda^{-1}: \lambda \in \Lambda\}$ is uniformly discrete. So the projection of $\Lambda$ to $G/E(\alpha)$ is locally finite. By Proposition \ref{Proposition: Intersection and projections approximate lattices w/ closed subgroups}, $\Lambda^2 \cap E(\alpha)$ is a uniform approximate lattice in $E(\alpha)$. 
\end{proof}

Considering inner automorphisms, Lemma \ref{Lemma: Intersection with centraliser} yields: 

\begin{lemma}[Corollary 6.4, \cite{machado2019goodmodels}]\label{Lemma: Intersection with centre}
Let $\Lambda$ be a uniform approximate lattice in an $S$-adic linear group $G$ and suppose that $\langle \Lambda \rangle$ is Zariski-dense. Then $Z_G \cap \Lambda^2$ is a uniform approximate lattice in $Z_G$. 
\end{lemma}

The following variations of Lemma \ref{Lemma: Intersection with centre} is particularly useful. 

\begin{lemma}\label{Lemma: Intersection with centraliser of kernel}
Let $\Lambda$ be an approximate lattice in an $S$-adic algebraic group $G$. Suppose that $\langle \Lambda \rangle$ is Zariski-dense in $G$. Let $p:G \rightarrow L$ be a regular factor map with $L$ reductive. Suppose that $\ker p \cap \langle \Lambda \rangle$ is Zariski-dense in $\ker p$. Define the subgroup $C_L$ by $ C_G(\ker p) \cap \Rad(G)$. Then $C_L \cap \Lambda^2$ is a uniform approximate lattice in $C_L$. 
\end{lemma}

\begin{proof}
According to the descending chain condition, there is a finite family $X=\{\gamma_1, \ldots, \gamma_n\} \subset \ker p$ such that $C_G(X) = C_G(\ker p)$. But $\Lambda^2 \cap \Rad(G)$ is a uniform approximate lattice. (see Theorem \ref{Theorem: Radical is hereditary} below).  So $C_G(\ker p) \cap \Rad(G) \cap \Lambda^2$ is an approximate lattice in $C_G(\ker p) \cap \Rad(G) $.  
\end{proof}

\begin{corollary}\label{Lemma: Intersection with centraliser of co-rank 1 subgroups}
With $\Lambda \subset G$ as above and $Z_{G,1} :=\bigcap_\pi C_G(\ker \pi)$ where $\pi$ runs through the rank one factors of $G$ such that $\pi(\Lambda)$ is contained in a discrete subgroup. Then $\Lambda^2 \cap Z_{G,1}$ is a uniform approximate lattice in $Z_{G,1}$.
\end{corollary}

\begin{proof}
Recall that $Z_{G,1}$ is defined as the subgroup of $\Rad(G)$ defined by
$$\langle C_G(N) \cap \Rad(G): \rank_S G/N \leq 1,\ N \cap \langle \Lambda \rangle \text{ Zariski-dense}\rangle.$$
But for $N$ as in the definition of $Z_{G,1}$, $C_G(N) \cap \Rad(G) \cap \Lambda^2$ is a uniform approximate lattice according to Lemma \ref{Lemma: Intersection with centraliser of kernel}. If $N_1,\ldots, N_r$ are such that  $\sum_{i=0}^r\left(C_G(N_i) \cap \Rad(G)\right) = Z_{G,1}$. Then $\Lambda^2 \cap Z_{G,1}$ is commensurable with $\sum_{i=0}^r\left(C_G(N_i) \cap \Rad(G) \cap \Lambda^2\right)$ which is a uniform approximate lattice (Proposition \ref{Proposition: Converse to Intersection and projections approximate lattices w/ closed subgroups}).
\end{proof}

\subsubsection{Solvable radicals}\label{Subsubsection: Solvable radical}
In \cite{machado2019goodmodels} we proved the following generalisation of both Bieberbach's and Mostow's results: 

\begin{theorem}[Theorem 1.9, \cite{machado2019goodmodels}]\label{Theorem: Intersection with amenable radical}
Let $\Lambda$ be an approximate lattice in a locally compact group $G$. Let $A$ be an amenable closed normal subgroup of $G$. Suppose that $G/A$ is the group of points of an almost simple algebraic group defined over a local field or a finite product of such groups. Suppose also that the projection of $\Comm_{G}(\Lambda)$ to each compact factor of $G/A$ is dense. Then $\Lambda^2 \cap A$ is an approximate lattice in $A$.
\end{theorem}

Specialising Theorem \ref{Theorem: Intersection with amenable radical} to $S$-adic linear groups yields information about the solvable radical. 

\begin{theorem}\label{Theorem: Radical is hereditary}
Let $\Lambda$ be an approximate lattice in an $S$-adic algebraic group $G$ that generates a Zariski-dense subgroup. Then $\Rad(G) \cap \Lambda^2$ is a uniform approximate lattice in $\Rad(G)$.
\end{theorem}

We prove below that similar properties hold for the nilpotent radical and other large characteristic subgroups.

\subsection{Intersection results for arithmetic subsets}\label{Subsection: S-adic algebraic groups} We defined in \S \ref{Subsection: Matrices with Pisot entries and the Borel--Harish-Chandra theorem for approximate lattices} an interesting family of uniformly discrete approximate subgroups by considering points of linear groups over PVS numbers. We can give another, more flexible, construction with an arithmetic flavour of uniformly discrete approximate subgroups of more general $S$-adic linear groups. 

\begin{definition}[Generalized arithmetic approximate subgroups]\label{Definition: GAAS}
 If $\mathbf{G}$ is a Zariski-connected linear group defined over $\mathbb{Q}$ and $G,H$ are two $S$-adic linear subgroups such that $G \times H \subset \mathbf{G}(\mathbb{A}_S)$, we define a \emph{generalized arithmetic approximate subgroup} (GAAS) as any approximate subgroup $\Lambda$ of $G$ commensurable with $p_{G}\left( \mathbf{G}(\mathbb{Z}_S) \cap G \times W_0\right)$ where $p_{G}: G \times H \rightarrow G$ is the natural projection, $W_0 \subset H$ is a symmetric relatively compact neighbourhood of the identity and $\mathbf{G}(\mathbb{Z}_S)$ is embedded diagonally in $\mathbf{G}(\mathbb{A}_S)$.
 \end{definition}

Generalized arithmetic approximate subgroups are uniformly discrete and extremely regular. In particular, every generalized arithmetic approximate subgroup is laminar. They also satisfy strong forms of intersection theorems, and it is this last property that is of particular interest to us.  

\begin{lemma}[Levi decomposition of GAAS]\label{Lemma: Levi decomposition of GAAS}
With the notation from Definition \ref{Definition: GAAS}. Let $L$ denote the Zariski-connected component of the identity of the Zariski-closure of $\mathbf{G}(\mathbb{Z}_S) \cap G \times H$. Then there is $\mathbf{L} \subset \mathbf{G}$ such that $L=\mathbf{L}(\mathbb{A}_S)$. Assume that the natural projections of $\mathbf{L}(\mathbb{A}_S)$ to $G$ and $H$ are open (equivalently, have finite index). 
There are $S, T$ and $U$ groups appearing in the Levi decomposition of $G$ - with $S$ a semi-simple Levi subgroup, $T$ the maximal torus centralized by $S$ and $U$ the unipotent radical - such that $\Lambda$ is commensurable with $$\left(\Lambda^2 \cap S\right)\left(\Lambda^2 \cap T\right)\left(\Lambda^2 \cap U\right).$$
Moreover, if $\Lambda$ is an approximate lattice, then  $\Lambda^2 \cap S$, $\Lambda^2 \cap T$ and $\Lambda^2 \cap U$ are approximate lattices in $S$, $T$ and $U$. 
\end{lemma}

\begin{proof}
Choose a Levi subgroup $\mathbf{S}$ of $\mathbf{L}$, let $\mathbf{T}$ denote the maximal torus it centralises and le $\mathbf{U}$ be the unipotent radical of $\mathbf{L}$. Define $S,T$ and $U$ as the Zariski-closures of the projections of $\mathbf{S}(\mathbb{A}_S)$, $\mathbf{T}(\mathbb{A}_S)$ and $\mathbf{U}(\mathbb{A}_S)$ to $G$. Then $S$ is semi-simple, $T$ is a torus centralising $S$, $U$ is a unipotent subgroup normalised by $ST$ and $STU$ has finite index in $G$. So $STU$ is a Levi decomposition of $G$. Similarly, the respective projections $S'$, $T'$ and $U'$ to $H$ provide a Levi decomposition of $H$. From the continuity of the Levi decomposition $H \rightarrow S' T' \times  U'$ and the almost-direct product decomposition $S' \times T' \rightarrow S'T'$ we deduce that there are compact neighbourhoods of the identity $W_{S'} \subset S', W_{T'} \subset T'$ and $W_{U'} \subset U'$ such that if $s \in S', t \in T', u \in U'$ are such that $stu \in W_0$,  then $s \in  W_{S'}$, $t \in W_{T'}$ and $u \in W_{U'}$.

Define $p_G: G \times H \rightarrow G$ and $p_H: G \times H \rightarrow H$ the natural projections. Define $\Lambda_S \subset S$, $\Lambda_T \subset T$ and $\Lambda_U \subset U$ as the GAAS defined using $W_{S'},W_{T'}$ and $W_{U'}$ respectively. Each of them is covered by finitely many translates of $\Lambda$. Conversely, if $\lambda \in \Lambda$, let $\gamma$ denote an element in $\mathbf{G}(\mathbb{Z}_S) \cap G \times W_0$ such that $p_G(\gamma)=\lambda$. There are now $s \in \mathbf{S}(\mathbb{Z}_S)$, $t \in \mathbf{T}(\mathbb{Z}_S)$ and $u \in \mathbf{U}(\mathbb{Z}_S)$ such that $stu=\lambda$. By the previous paragraph, $p_H(s) \in W_{S'}$, $p_H(t) \in W_{T'}$ and $p_H(u) \in W_{U'}$. Hence, $\lambda \in \Lambda_S\Lambda_T\Lambda_U$. So the first part of Lemma \ref{Lemma: Levi decomposition of GAAS} is proved.

Suppose that $\Lambda$ is an approximate lattice to begin with. Then the projections to both $S$ and $T$ are uniformly discrete, so they are approximate lattices by Proposition \ref{Proposition: Intersection and projections approximate lattices w/ closed subgroups}. Since the kernel of the projection to $ST$ is $U$, $\Lambda^2 \cap U$ is an approximate lattice as well by Proposition \ref{Proposition: Intersection and projections approximate lattices w/ closed subgroups}.
\end{proof}

\subsection{Arithmeticity from action on subgroups}

To prove that the nilpotent radical is hereditary, we will in fact prove and harness an arithmeticity statement. Our strategy is as follows: we first show that the approximate lattice considered intersects \emph{some} non-trivial normal unipotent subgroup in an approximate lattice; we then use the above extension of Meyer's theorem (Proposition \ref{Proposition: Meyer's theorem $S$-adic unipotent group}) to show that this intersection has an arithmetic origin; finally, we exploit the conjugation action of the whole approximate lattice on the unipotent subgroup to show that that arithmetic structure propagates to a large quotient of $G$ which allows us to conclude.

\begin{proposition}\label{Proposition: Arithmeticity from action on subgroups}
Let $\Lambda$ be an approximate lattice in a unipotent $S$-adic group $U$. Let $G \subset \Aut(U)$ be a Zariski-closed subgroup and let $\Xi \subset G$ be an approximate subgroup of $\Aut(U)$ such that $\bigcup_{\xi \in \Xi} \xi(\Lambda)$ is commensurable with $\Lambda$. We have:
\begin{enumerate}
\item if $\langle \Xi \rangle$ is Zariski-dense in $G$, then $\Xi$ is contained in a generalized arithmetic approximate subgroup of $G$;
\item if $\Xi$ is an approximate lattice of $G$, then $\Xi$ itself is a generalized arithmetic approximate subgroup of $G$;
\item Let $\Lambda$ be an approximate lattice in an $S$-adic affine algebraic group $G$. Suppose that $\Rad(G)=U$. If there is a Levi subgroup $S \subset G$ such that $\Lambda^2 \cap S$ is an approximate lattice, then $\Lambda$ is commensurable with the almost direct product of a generalized arithmetic approximate lattice and a product of rank one lattices. 
\end{enumerate}
\end{proposition}

In the statement of Proposition \ref{Proposition: Arithmeticity from action on subgroups}, the notation $\Aut(U)$ refers to all \emph{regular} automorphisms. Hence, $\Aut(U)$ has a natural $S$-adic algebraic structure.

\begin{proof}
Let $\mathbf{U}, U_1, U_2, \pi$ be given by Proposition \ref{Proposition: Meyer's theorem $S$-adic unipotent group}. Identify $U_1$ and $U$. Suppose moreover that $\Lambda$ is the model set associated with $(U_1, U_2, \mathbf{U}(\mathbb{Z}_S))$.  By Corollary \ref{Corollary: Bounded approximate subgroups of of automorphisms are relatively compact} upon considering a commensurable approximate subgroup $\Xi'$, we may also assume that $\Xi$ normalises $\langle \Lambda \rangle$. 

We will use the strong relation between unipotent groups and their Lie algebras in characteristic $0$ (see \cite[IV.2.4]{zbMATH03670601} for references and \cite[II]{raghunathan1972discrete} for the special cases of $\mathbb{R}$ and $\mathbb{Q}$). Let $\mathfrak{u}$ be the $\mathbb{Q}$-Lie algebra of $\mathbf{U}$. Let $\mathfrak{u}_S$ denote $\mathfrak{u} \otimes \mathbb{A}_S$. Write $\exp$ and $\log$ for the usual maps. Recall that since $\mathbf{U}$ is unipotent $\exp$ and $\log$ are $\mathbb{Q}$-regular isomorphisms. Conjugating automorphisms using $\exp$ we can identify the automorphism groups of $\mathbf{U}$ with a closed subgroup $\mathbf{L}$ of $\GL(\mathfrak{u}_S)$. In particular, $\mathbf{L}(\mathbb{Q}) = \Aut_{\mathbb{Q}}(\mathbf{U})$  and $\mathbf{L}(\mathbb{A}_S) = \Aut_{\mathbb{A}_S}(\mathbf{U})$. 

 By part (5) of Proposition \ref{Proposition: Meyer's theorem $S$-adic unipotent group} we have a group homomorphism $\Xi \rightarrow \Aut_{\mathbb{A}_S}(\mathbf{U})$ sending $\xi$ to $\xi_{\mathbf{U}}$. Choose $e_1, \ldots, e_n$ a $\mathbb{Q}$-basis of $\mathfrak{u}$ and use it to identify $\GL(\mathfrak{u}_S)$ and $\GL_n(\mathbb{A}_S)$. Since $\langle \Lambda \rangle$ is covered by finitely many cosets of $\mathbf{U}(\mathbb{Z}_S)$ and $\log$ is $\mathbb{Q}$-regular, there is an integer $m_1 > 0$ such that the coefficients of elements of $\log \langle \Lambda \rangle$ in the basis $(e_1, \ldots, e_n)$ are contained in $\frac{1}{m}\mathbb{Z}_S$.  So the $\mathbb{Z}_S$-span of $\log \langle \Lambda \rangle$ is finitely generated. Since $\log \langle \Lambda \rangle$ spans $\mathfrak{u}$ - as it is Zariski-dense - and $\mathbb{Z}_S$ is a principal ideal domain, we may assume that $e_1, \ldots, e_n$ are such that the $\mathbb{Z}_S$ span of $\log \langle \Lambda \rangle$ is $\bigoplus_{i=1}^n \mathbb{Z}_S e_i$.  Then $\bigoplus_{i=1}^n \mathbb{Z}_S e_i$ is stable under the action of $\Xi$. So $\xi_{\mathbf{U}} \in \GL_n(\mathbb{Z}_S)$ i.e. the image $\Xi_{\mathbf{U}}$ of $\Xi$ in $\mathbf{L}$ is contained in the $S$-arithmetic group $\mathbf{L}(\mathbb{Z}_S)$ (remark that this potential change of basis does not change the $\mathbb{Q}$-structure). 

By (5) of Proposition \ref{Proposition: Meyer's theorem $S$-adic unipotent group} again, the image of $\langle \Xi \rangle$ through that map is thus contained in $\Aut(U_1) \times \Aut(U_2) \subset \Aut_{\mathbb{A}_S}(\mathbf{U})$. Here we have identified $\Aut(U_1)$ with $\{\alpha: \mathbf{U} \rightarrow \mathbf{U}: \alpha(U_1) \subset U_1, \alpha_{|U_2}=\id_{U_2}\}$ and symmetrically for $\Aut(U_2)$. That way we see $\Aut(U_1)$ and $\Aut(U_2)$ as $S$-adic Zariski-closed subgroups of $\mathbf{L}(\mathbb{A}_S)$. 

By assumption for all $m \geq 0$,  $X=\bigcup_{\xi \in \Xi} \xi (\Lambda^m)$ is an approximate subgroup commensurable with $\Lambda$. But the projection of $\Lambda$ to $U_2$ generates a Zariski-dense subgroup and is contained in a compact subset of $U_2$. Therefore,  there is $m \geq 0$ such that the projection of $\Lambda^m$ is dense in a compact subset $K \subset U_2$ such that $\log(K)$ spans the Lie algebra of $U_2$. Moreover, $\xi_{\mathbf{U}}(K)$ is relatively compact for all $\xi \in \Xi$. Therefore, the projection of $\Xi_{\mathbf{U}}$ to $\Aut(U_2)$ is relatively compact. All in all, $\Xi_{\mathbf{U}}$ is contained in $\mathbf{L}(\mathbb{Z}_S) \cap \left(\Aut(U_1) \times \Aut(U_2)\right)$ and projects to a relatively compact subset of $\Aut(U_2)$. So $\Xi$ is contained in a generalized arithmetic approximate subgroup. Hence, (1) is satisfied.

If, moreover, $\Xi$ is an approximate lattice in $G$, then it must be commensurable with the generalized arithmetic approximate subgroup it is contained in by Corollary \ref{Corollary: Towers of star-approximate subsets are commensurable}. This proves (2).


It remains finally to prove (3). Suppose first that the action of $S$ on $U$ is faithful.  We have that $\Lambda^2 \cap U$ is an approximate lattice in $U$ by Theorem \ref{Theorem: Radical is hereditary}. Since $S$ acts faithfully on $U$, the conjugation action identifies $S$ with a closed subgroup of $\Aut(U)$. Moreover, we can apply (2) to $\Xi$ the image of $\Lambda^2 \cap S$ in $\Aut(U)$ acting on $\Lambda^2 \cap U$. With $U_1=U$ and $U_2$ as above we see that $G$ can be identified as a Zariski-closed subgroup of $\Aut(U) \ltimes U$ with the usual structure. Moreover, it is a generalized arithmetic approximate subgroup when we consider the product $G \times  \left(\Aut(U_2) \ltimes U_2\right)$ as a subset of  $\Aut_{\mathbb{A}_S}(\mathbf{U})\ltimes \mathbf{U} ( \mathbb{A}_S)$.

If now the action of $S$ on $U$ has a kernel $S_0$, (1) implies that the projection of $\Lambda^2 \cap S$ to $S/S_0$ is uniformly discrete. So there is $S_1 \subset S$ such that $S$ is the almost direct product of $S_0$ and $S_1$ and $\Lambda^2 \cap S_0$ is an approximate lattice in $S_0$, $\Lambda^2 \cap S_1$ is an approximate lattice in $S_1$ and $\Lambda$ is commensurable with $(\Lambda^2 \cap S_0)(\Lambda^2 \cap S_1)$. The previous paragraph implies that $(\Lambda^2 \cap S_1)(\Lambda^2 \cap U)$ is a generalized arithmetic approximate subgroup and $(\Lambda^2 \cap S_0)$  is commensurable with  the almost direct product of an \emph{arithmetic} approximate lattice and a product of rank one lattices \cite[Thm 7.4]{hrushovski2020beyond}. Since $(\Lambda^2 \cap S_0)(\Lambda^2 \cap S_1)(\Lambda^2 \cap U)$ is commensurable with $\Lambda$, (3) is proved.
\end{proof}

The first implication concerns the commensurator of certain approximate lattices.

\begin{corollary}\label{Corollary: Unipotent rational points in commensurator}
Let $\Lambda$ be an approximate lattice in an $S$-adic linear group $G$. Let $U \subset G$ be a normal Zariski-closed unipotent subgroup such that $\Lambda^2 \cap U$ is an approximate lattice in $U$. Let $\mathbf{U}$ be as in Proposition \ref{Proposition: Meyer's theorem $S$-adic unipotent group}. Then the image of $\mathbf{U}(\mathbb{Q})$ in $U$ belongs to the commensurator of $\Lambda$. 
\end{corollary}

\begin{proof}
Write $U_1 \times U_2 = \mathbf{U}(\mathbb{A}_{S})$ as in Proposition \ref{Proposition: Meyer's theorem $S$-adic unipotent group}. Identify $U_1$ and $U$. By Proposition \ref{Proposition: Arithmeticity from action on subgroups} we see that the group homomorphism $\phi: \Lambda \rightarrow \Aut_{\mathbb{A}_{S}}(\mathbf{U})$ takes values in $\Aut(U_1) \times \Aut(U_2)$. Moreover, there is an approximate subgroup $\Lambda'$ commensurable with $\Lambda$, $\Lambda'$ that normalizes $\mathbf{U}(\mathbb{Z}_S)$ (Corollary \ref{Corollary: Bounded approximate subgroups of of automorphisms are relatively compact}) and the projection of $\phi(\Lambda')$ to $\Aut(U_2)$ is relatively compact.  Choose $u$ any element in $\mathbf{U}(\mathbb{Q})$. The set of commutators $[\phi(\Lambda'), u]:=\{[\phi(\lambda),u] : \lambda \in \Lambda'\}$ projects to a compact subset in $U_2$.  Moreover, $$\log [\phi(\Lambda'), u] \subset \frac{1}{m} span_{\mathbb{Z}_S} \log \mathbf{U}(\mathbb{Z}_S)$$ for some integer $m > 0$. So $[\Lambda', u']$ is covered by finitely many right translates of $\Lambda^2 \cap U$ where $u'$ is the projection of $u$ to $U_1$. But $\Lambda^u \subset \Lambda'[\Lambda',u]$.  So $\Lambda^u$ is covered by finitely many translates of $u$. 
\end{proof}

\subsection{The nilpotent radical is hereditary}
We turn now to the first new heredity result for general approximate lattices. It concerns the nilpotent radical of a group. 

\begin{proposition}\label{Proposition: Unipotent radical is hereditary}
Let $\Lambda$ be a uniform approximate lattice in a solvable $S$-adic linear group $R$. Suppose that $\Lambda$ generates a Zariski-dense subgroup. Let $N$ be the maximal Zariski-closed Zariski-connected nilpotent group: 
\begin{enumerate}
\item $N= T \times U$ where $T$ is the maximal central torus and $U$ is the unipotent radical of $R$; 
\item $\Lambda^2 \cap N$ is a uniform approximate lattice in $N$.
\end{enumerate}
\end{proposition}

\begin{proof}
Since $N$ is normal and $T$ is characteristic, $T$ is a normal torus in $G$. Hence, $T$ is central. Now, Part (1) is \cite[]{MR1102012}.

To prove (2) let us first suppose that $T=\{e\}$. We will use the following fact multiple times: if $s$ is a semi-simple element of $R$ fixing $U$ point-wise, then $s$ is central i.e. $s=e$. Indeed, if $s$ is semi-simple, then $s$ is contained in a maximal torus $S$. In particular, $s$ commutes with all elements of $S$ and all elements of $U$. But $R=SU$,  so $s$ is indeed central. 

Let $R' \subset R$ be the minimal Zariski-closed subgroup of $R$ containing $U$ such that $\Lambda^2 \cap R'$ is a uniform approximate lattice. If $s$ is a semi-simple element in $Z_{R'}$, then $s$ fixes $U$ point-wise. So $s=e$. In other words, just like $R$, $R'$ contains no central semi-simple element.  We can thus assume from now on that $R'=R$. 

Since $\Gamma=\langle \Lambda \rangle$ is Zariski-dense, we have $C_R([R,R])=C_R([\Gamma, \Gamma])$. By Lemma \ref{Lemma: Intersection with centraliser}, $\Lambda^2 \cap C_R([R,R])$ is therefore a uniform approximate lattice in $C_R([R,R])$. But $C_R([R,R])$ is normal - and we will show that under our assumption it is unipotent. We know that $[R,R] \subset U$. Take $s \in C_R([R,R])$ semi-simple, then conjugation by $s$ acts trivially on $R/[R,R]$. So it acts trivially on $U/[R,R]$. Moreover, conjugation by $s$ is trivial on $[R,R]$ by assumption. Since $s$ is semi-simple, we deduce that $s$ acts trivially on $U$. So $s=e$. 

Write $U' = C_R([R,R])$. Then $U'$ is a normal subgroup of $U$ and $\Lambda^2 \cap U'$ is a uniform approximate lattice in $U'$.  Considering the action of $R$ on $U'$ by conjugation and the approximate subgroups $\Lambda$ and $\Lambda^2 \cap U'$, we see by Proposition \ref{Proposition: Arithmeticity from action on subgroups} that the projection $\Lambda'$ of $\Lambda$ to $R/C_R(U')$ is a generalized arithmetic approximate subgroup.  It is also an approximate lattice by Proposition \ref{Proposition: Intersection and projections approximate lattices w/ closed subgroups}.   Let $U''$ be the unipotent radical of $R/C_R(U')$. The image of $U$ is contained in $U''$ and $U'' \cap \Lambda'^2$ is a uniform approximate lattice (because $\Lambda'$ is a generalized arithmetic approximate lattice). Let $R'$ denote the pull-back of $U''$. Then $U \subset R'$ and $R' \cap \Lambda^2$ is a uniform approximate lattice (Proposition \ref{Proposition: Converse to Intersection and projections approximate lattices w/ closed subgroups}). So $R' = R$ by minimality. In particular, any semi-simple element of $R$ belongs to $C_R(U')$. As a consequence all semi-simple elements are trivial by the previous paragraph i.e. $R=U$. 

Let us now consider the case $T\neq \{e\}$. We can proceed by induction. Since $T \neq \{e\}$, $Z_R$ is non-trivial. The projection $N'$ of $N$ to $R/Z_R$ is the maximal nilpotent group of $R/Z_R$. Moreover, by Lemma \ref{Lemma: Intersection with centre}, the projection $\Lambda'$ of $\Lambda$ to $R/Z_R$ is an approximate lattice. By induction, $\Lambda'^2 \cap N'$ is an approximate lattice in $N'$. So $\Lambda^2 \cap N$ is an approximate lattice in $N$ (Proposition \ref{Proposition: Converse to Intersection and projections approximate lattices w/ closed subgroups}). This proves (2).
\end{proof}

Proposition \ref{Proposition: Unipotent radical is hereditary} enables the use of partial arithmeticity theorems. 

\begin{proposition}\label{Proposition: Arithmeticity from action on nilpotent subgroups}
Let $\Lambda$ be an approximate lattice in a unipotent $S$-adic group $G$. Suppose that $\Lambda$ is Zariski-dense and let $N$ be a normal Zariski-closed Zariski-connected nilpotent $S$-adic linear subgroup containing the centre of $G$.  Suppose that $\Lambda^2 \cap N$ is an approximate lattice in $N$.  Then the projection of $\Lambda$ to $G/C_G(N)$ is GAAS. 
\end{proposition}

\begin{proof}
As in the proof of (1) of Proposition \ref{Proposition: Unipotent radical is hereditary}, $N = U \times T$ where $T$ is the maximal torus contained in $N$ and $T$ is central in $G$. Write $\Lambda_N:=\Lambda^2 \cap N$ and choose a symmetric compact neighbourhood of the identity $W \subset T$. Define $\Lambda_U:=p_U\left(\Lambda_N^2 \cap \left( W \times U\right)\right)$. Since $T$ is central, for every $\lambda \in \Lambda$, 
$$\lambda \Lambda_U \lambda^{-1} = p_U\left(\lambda\Lambda_N^2\lambda^{-1} \cap \left( W \times U\right)\right).$$
But $\Lambda$ commensurates $\Lambda_N$ (Lemma \ref{Lemma: Intersection of approximate subgroups} (3)), so $\Lambda$ commensurates $\Lambda_U$ (Lemma \ref{Lemma: Intersection of approximate subgroups} (3) again). So the image $\Lambda'$ in $\Aut(N)$ via conjugation is as GAAS of the image $G'$ of $G$. But the kernel of the conjugation action of $G$ on $N$ is $C_G(N)$, so $G'$ is indeed $G/C_G(N)$. This proves Proposition \ref{Proposition: Arithmeticity from action on nilpotent subgroups}.
\end{proof}
%

\subsection{A hereditary unipotent subgroup}

It will also matter that some unipotent (not only nilpotent) subgroup is hereditary. Although the unipotent radical may fail to be hereditary, we still prove that a certain large characteristic unipotent subgroup is hereditary. 

\begin{proposition}\label{Lemma: Intersection with the derived subgroup}
Let $G$ be a Zariski-connected $S$-adic group and $\Lambda$ be an approximate lattice in $G$ generating a Zariski-dense subgroup. Let $U$ be the maximal normal unipotent subgroup of $G$ and define the Zariski-closed subgroup
$$[G,U]:=\overline{\langle[g,u] : g\in G, u\in U\rangle}.$$
Then $\Lambda^2 \cap [G,U]$ is an approximate lattice in $[G,U]$. 
\end{proposition}

We remark here that $[G,U]=[G,R]$.
\begin{proof}

We will proceed by induction on the dimension of $U$. Let $T$ be the maximal central torus in $\Rad(G)$ and define $N=T \times U$. According to Theorem \ref{Theorem: Radical is hereditary} and Proposition \ref{Proposition: Unipotent radical is hereditary}, $\Lambda_N:=\Lambda^2 \cap N$ is an approximate lattice in $N$.

 We will first show that $\Lambda^2 \cap [U,U]$ is an approximate lattice. If $U$ is abelian there is nothing to prove. So suppose that $U$ is not abelian and, moreover, that $T=\{e\}$ i.e. $N=U$. Write $C_1=Z_U, C_2, \ldots, C_n=U$ the ascending central sequence of $U$. Recall that $[U,C_i] \subset C_{i-1}$. because $U$ is not abelian, $C_2 \setminus C_1$ is not empty. As a simple consequence of the Meyer-type Proposition \ref{Proposition: Meyer's theorem $S$-adic unipotent group}, there is in fact $\xi \in \Lambda$ such that $\xi \in C_2 \setminus C_1$. For every $u,v \in U$, since $[\xi, u]$ and $[\xi,v]$ are both central, we have $[\xi,uv]=[\xi,u][\xi,v]$. Therefore, $\phi_{\xi}: u \mapsto [\xi,u]$ is a regular homomorphism of $U$. It has non-trivial range since $\xi \notin C_1$. Let $U_{\xi}$ denote the range of $\phi_{\xi}$. Since $\xi \in \Lambda$, we have $\phi_{\xi}(\Lambda_N) \subset \Lambda^4$. Hence, $\phi_{\xi}(\Lambda_N)$ is a relatively dense subset contained in $\Lambda^4 \cap U_{\xi}$. So $\Lambda^2 \cap U_{\xi}$ is an approximate lattice in $U_{\xi}$ (Lemma \ref{Lemma: Intersection of approximate subgroups}).
 
 Suppose now that $T$ is not trivial any more. Define $\Lambda' \subset U$ as the projection to $U$ of $\Lambda^4 \cap \left(W \times U\right)$ where $W$ denotes a symmetric compact neighbourhood of the identity in $T$. Then $\Lambda'$ is an approximate lattice in $U$. Hence, the previous paragraph provides $\xi \in \Lambda'$ such that $\phi_{\xi}:u \mapsto [\xi,u]$ is a regular homomorphism and the image $[\xi,\Lambda']$ is an approximate lattice in the range $U_{\xi}$ of $\phi_{\xi}$. For every $\lambda \in \Lambda'$ choose $\tilde{\lambda} \in \Lambda^4 \cap N$ such that $\tilde{\lambda}$ projects to $\lambda'$ in $U$. Similarly, choose $\tilde{\xi} \in \Lambda^4 \cap N$ such that $\tilde{\xi}$ projects to $\xi$ in $U$. Then 
 $$[\xi,\lambda] = [\tilde{\xi}, \tilde{\lambda}].$$
 So $[\tilde{\xi},\Lambda^4 \cap N] $ contains $[\xi, \Lambda']$ which is an approximate lattice in $U_{\xi}$. Since $[\tilde{\xi},\Lambda^4 \cap N] \subset \Lambda^{16} \cap U_{\xi}$, we have that $\Lambda^2 \cap U_{\xi}$ is an approximate lattice in $U_{\xi}$ (Lemma \ref{Lemma: Intersection of approximate subgroups}). 
 
Therefore, we have shown that whenever $U$ is not abelian, we can find a non-trivial unipotent central subgroup $U' \subset [U,U]$ such that $\Lambda^2 \cap U'$ is an approximate lattice in $U'$. Using Proposition \ref{Proposition: Intersection and projections approximate lattices w/ closed subgroups}, we can proceed by induction and show that $\Lambda^2 \cap [U,U]$ is an approximate lattice in $[U,U]$.

Let us turn back to the proof of Proposition \ref{Lemma: Intersection with the derived subgroup}. Since $[U,U] \cap \Lambda^2$ is an approximate lattice, it is enough to prove the result in $G/[U,U]$ (Proposition \ref{Proposition: Intersection and projections approximate lattices w/ closed subgroups}). Thus, we assume from now on that $U$ is abelian. Since $U$ is abelian, we find that for all $\xi \in \langle\Lambda\rangle$, the map 

\begin{align*}
 \phi_{\xi}: & T \times U \rightarrow U \\
  & (t,u) \rightarrow [\xi, (t,u)]
 \end{align*}
is a regular homomorphism. Let $U_{\xi}$ denote its range. Then $\phi_{\xi}(\Lambda^2 \cap N)$ is relatively dense in $U_{\xi}$. Since $\langle \Lambda \rangle$ is Zariski-dense, one can find $\xi_1, \ldots, \xi_r \in \Lambda^m$ for some integer $m>0$ such that $\sum_{i=1}^r U_{\xi_i} = [G,U]$ (recall that $U$ is an $S$-adic abelian unipotent group). Therefore, $\sum_{i=1}^r \phi_{\xi_i}(\Lambda^2 \cap N) = \sum_{i=1}^r [\xi_i,\Lambda]$ is relatively dense in $[G,U]$. But for all $i \in \{1, \ldots, r\}$, we have $[\xi_i, \Lambda] \subset \Lambda^{2m+4}$. So $\sum_{i=1}^r [\xi_i,\Lambda] \subset \Lambda^{r(2m+4)}$. According to Lemma \ref{Lemma: Intersection of approximate subgroups} again, $[G,U] \cap \Lambda^2$ is an approximate lattice in $[G,U]$. 
\end{proof}

\begin{remark}\label{Remark: Alternative charac. derived subgroup}
The group $[G,U]$ can also be characterised as follows: it is the smallest normal unipotent subgroup $U'$ such that $G/U'$ is a \emph{direct} product of a reductive group and an abelian unipotent group. As a consequence of that fact if $G_1 \rightarrow G_2$ is an homomorphism and $U_1 \subset G_1$, $U_2 \subset G_2$ are the respective unipotent radicals, then the image of $[G_1,U_1]$ is contained in $[G_2,U_2]$. 
\end{remark}

\subsection{Intersection theorem for isotypic factors}
We mention a related result for groups with abelian radical.
\begin{proposition}\label{Proposition: Precise decomposition quasi-cocycle}
Let $G= L\ltimes A$ be an $S$-adic linear group with $L$ reductive and $A$ abelian. Let $\Lambda \subset G$ be a Zariski-dense approximate lattice such that $\Lambda_r:=\Lambda^2 \cap A$ is an approximate lattice and let $\Lambda_{red}$ denote the projection of $\Lambda$ to $L$. Write $A = T \times U$ with $T$ a torus and $U$ unipotent (Proposition \ref{Proposition: Unipotent radical is hereditary}). For any normal subgroup $N$ of $\langle \Lambda_{red} \rangle$ write $U^N$ the sum of all irreducible subrepresentations of $U$ whose kernel is $N$. Then there is a family $\mathcal{N}$ of normal subgroups of $\langle \Lambda_{red} \rangle$ such that $A = Z_G \oplus \bigoplus_{N \in \mathcal{N}} U^N$. Moreover, $U^N \cap \Lambda^2 $ is an approximate lattice in $U^N$ for all $N \in \mathcal{N}$, the map that sends $N$ to its Zariski-closure $\bar{N}$ in $L$ is one-to-one and $\Lambda_{red}^2 \cap \bar{N}$ is an approximate lattice for all $N \in \mathcal{N}$. 
\end{proposition}

\begin{proof}
We know by Lemma \ref{Lemma: Intersection with centre} that $Z_G \cap \Lambda^2$ is an approximate lattice. Since $L$ is reductive, $Z_G = T \times \left(Z_G \cap U\right)$ and $Z_G \cap U$ is the isotypic factor corresponding to the trivial representation, we know that there is a unique  $U' \subset U$ Zariski-closed and normalised by $L$ such that $\left( Z_G \cap U \right)\oplus U' = U$. For every $\gamma \in \langle\Lambda_{red}\rangle$, we have that $(\gamma - id)\cdot (Z_G) = \{0\}$ and $(\gamma - id)\cdot (U') \subset U'$ where $\cdot$ denotes the action of $L$ on $A$. If we write $U''$ the subspace generated by the spaces $(\gamma - id)\cdot (U')$ when $\gamma$ ranges in $\langle \Lambda_{red} \rangle$, then $U''$ is a subrepresentation.  We show that$U'' = U'$. If now $I \subset U'$ is any irreducible representation, then there is $\gamma \in \langle\Lambda_{red}\rangle$ that acts non-trivially on $I$ since $\langle\Lambda_{red}\rangle$ is Zariski-dense in $L$ (\cite[A.11]{hrushovski2020beyond}). So $I \cap U'' \supset (\gamma - id)\cdot  I \neq \{0\}$. So $I \subset U''$, which yields $U''=U'$. Since $U'$ is finite-dimensional, there are $\gamma_1, \ldots, \gamma_r \in \langle \Lambda_{red} \rangle$ such that 
$$\sum_{i=1}^r(\gamma_i-id)(A)=\sum_{i=1}^r(\gamma_i-id)(U') = U'.$$ 
As a consequence,  one sees that $\Xi:=\sum_{i=1}^r(\gamma_i-id)(\Lambda_r)$ is relatively dense in $U'$. But $(\gamma_i - id)(\Lambda_r)$ is contained in $\gamma_i (\Lambda_r) + \Lambda_r$. So it is covered by finitely many translates of $\Lambda_r$.  Thus, $\Xi$ is a uniform approximate lattice. But $\Xi$ is covered by finitely many translates of $\Lambda_r^2 \cap U'$ (Lemma \ref{Lemma: Intersection of commensurable sets}). In other words, $\Lambda_r^2 \cap U'$ is an approximate lattice in $U'$. 

It remains now to prove the result in $U'$. But for any $N \in \mathcal{N}$ maximal for inclusion, Lemma \ref{Lemma: Intersection with centraliser of co-rank 1 subgroups} implies that $\Lambda_r^2 \cap \left(C_G(N) \cap U'\right)$ is an approximate lattice. Since $N$ is normal, $\left(C_G(N) \cap U'\right)$ is Zariski-closed and normal in $L \ltimes A$.  As above, we may find $U'''$ Zariski-closed and normal such that $U' = \left(C_G(N) \cap U'\right) \oplus U'''$. Considering the linear maps $\gamma - id$ as above for $\gamma \in  N$ we conclude similarly, that $\Lambda_r^2 \cap U'''$ is an approximate lattice as well. Applying inductively this process in $U'''$ yields the first part of Proposition \ref{Proposition: Precise decomposition quasi-cocycle}. 

Let us finally prove the last part of the conclusion. Notice that since the action of $\langle \Lambda_{red} \rangle$ comes from the action of $L$ which is algebraic, if $\bar{N_1}=\bar{N_2}$ for $N_1, N_2 \in \mathcal{N}$, then any irreducible representation present in $U^{N_1}$ is trivial on $N_2$ and conversely. So $U^{N_1}=U^{N_2}$ and, thus, $N_1=N_2$. Moreover, looking at the action of $\Lambda_{red}$ on $\Lambda^2 \cap U^N$ for some $N \in \mathcal{N}$ we see that $\Lambda_{red}(\Lambda^2 \cap U^N) \subset \Lambda^4 \cap U^N$ which is uniformly discrete. So the projection of $\Lambda_{red}$ to $L/N'$ where $N'$ denotes the kernel of the action of $L$ on $U^N$ is uniformly discrete. According to Proposition \ref{Proposition: Intersection and projections approximate lattices w/ closed subgroups}, $\Lambda_{red}^2 \cap N'$ is an approximate lattice and its Zariski-closure contains the Zariski-connected component of the identity of $N'$ by \cite[A.11]{hrushovski2020beyond}. But $\bar{N} \subset N'$ and $\Lambda_{red}^2 \cap N'\subset  N$. So $\bar{N}$ has finite index in $N'$ which implies that $\Lambda^2 \cap \bar{N}$ is an approximate lattice in $\bar{N}$.  The result is proved.
\end{proof}

\subsection{About intersection theorems for non-solvable groups}
We conclude with a simple consideration concerning intersection theorems for non-solvable and non-normal subgroups.  We prove that Levi subgroups that intersect an approximate lattice in an approximate lattice are pairwise conjugate by a commensurating element. 

\begin{corollary}\label{Corollary: Choice of Levi subgroup is irrelevant}
 Let $\Lambda$ be an approximate lattice in a Zariski-connected $S$-adic affine algebraic group $G$ that is Zariski-dense. Let $U$ be the unipotent radical of $G$. Suppose that $G$ contains no non-trivial normal semi-simple subgroup. Suppose that $L_1, L_2$ are Zariski-closed subgroups of $G$ such that: 
 \begin{enumerate}[label=(\roman*)]
 \item either, $L_1\cap \Rad(G) \subset L_2 \cap \Rad(G)$, $L_1\cap \Rad(G)$ is normal in $G$ and both $L_1$ and $L_2$ contain a semi-simple Levi subgroup;
 \item or, $L_1\cap U \subset L_2 \cap U$, $L_1\cap \Rad(G)$ is normal in $G$ and both $L_1$ and $L_2$ contain a reductive Levi subgroup.
 \end{enumerate}
 If $L_1 \cap \Lambda^2$ and $L_2 \cap \Lambda^2$ are approximate lattices in $L_1$ and $L_2$ respectively, then there is $\gamma \in \Comm_G(\Lambda) \cap U$ such that $L_1^{\gamma} \subset L_2$.
\end{corollary}

\begin{proof}
We will treat case (i) only - case (ii) follows from the same arguments. We assume first that $\Rad(G) \cap L_1$ is trivial. Note that if $S$ is any semi-simple Levi subgroup, then $\tilde{G}:=S[G,U]$ is a normal subgroup of $G$ that contains all semi-simple Levi subgroups of $G$ (Remark \ref{Remark: Alternative charac. derived subgroup}).  In particular, $\tilde{G}$ contains $L_1$. Consider $L_2':=L_2 \cap S[G,U]$. It must contain a semi-simple Levi subgroup and $\Lambda^2 \cap L_2'$ is an approximate lattice in $L_2'$ (Corollary \ref{Lemma: Stability approximate lattices under intersection ambient group} and Proposition \ref{Lemma: Intersection with the derived subgroup}). So $L_1$ and $L_2'$ also satisfy $(i)$. Whence, we may assume $L_2 \subset S[G,U]$ from the start. Now,  $L_1$ is a Levi subgroup of $\tilde{G}$ and $\Lambda^2 \cap L_1$ is an approximate lattice in $L_1$.  So according to part (3) of Proposition \ref{Proposition: Arithmeticity from action on subgroups}, $\Lambda^2 \cap \tilde{G}$ is a generalized arithmetic approximate lattice. Take $\mathbf{H} \subset \GL_n$ a $\mathbb{Q}$-subgroup, $G_1 \times G_2 \subset \mathbf{H}(\mathbb{A}_S)$ a product of two Zariski-closed subgroups such that $G_1$ can be identified with $G$ via a regular homomorphism, $(G_1,G_2, (G_1 \times G_2) \cap \mathbf{H}(\mathbb{Z}_S))$ is a cut-and-project scheme and $\Lambda$ is commensurable with a model set coming from it. We assume that $\Lambda$ is equal to the said model set. Let furthermore $\mathbf{G}$ denote the Zariski-closure of $(G_1 \times G_2) \cap \mathbf{H}(\mathbb{Z}_S)$. For $i \in \{1,2\}$ let $\Gamma_i$ denote $\left(L_i \times G_2\right) \cap \mathbf{H}(\mathbb{Z}_S)$ and $\mathbf{L}_i$ denote the Zariski-closure of $\Gamma_i$. The projection of $\mathbf{L}_i(\mathbb{A}_S)$ to $L_i$ contains $\langle \Lambda \rangle \cap L_i$. By our density assumption on $\Lambda$ and the Borel density theorem \cite[A.11]{hrushovski2020beyond}, we have that $\mathbf{L}_1(\mathbb{A}_S)$ projects surjectively to $L_1$ and the projection of $\mathbf{L}_2(\mathbb{A}_S)$ contains a semi-simple Levi subgroup of $L_2$. Since the projection to $G_1$ is injective on $\mathbf{G}(\mathbb{Q})$, we have that $\mathbf{L}_1$ is semi-simple. We will show in addition that $\mathbf{L}_1$ is a semi-simple Levi subgroup of $\mathbf{G}$. Indeed, there is a semi-simple subgroup $\mathbf{L}_1'$ of $\mathbf{G}$ such that $\mathbf{L}_1'$ is normalised by $\mathbf{L}_1$, $\mathbf{S}=\mathbf{L}_1\mathbf{L}_1'$ is a semi-simple Levi subgroup of $\mathbf{G}$ and $\mathbf{L}_1(\mathbb{A}_S) \cap \mathbf{L}_1'(\mathbb{A}_S)$ is finite. The projection of $\mathbf{S}(\mathbb{A}_S)$ to $G_1$ is semi-simple and contains $L_1$. So it is equal to $L_1$. By surjectivity of $\mathbf{L}_1(\mathbb{A}_S) \rightarrow L_1$, we find $\mathbf{L}_1'(\mathbb{A}_S) \subset G_2$. But, again, the projection to $G_1$ is injective on $\mathbf{G}(\mathbb{Q})$. So $\mathbf{L}_1'(\mathbb{Q})=\{e\}$ i.e. $\mathbf{L}_1$ is already a semi-simple Levi subgroup. A similar argument tells us that $\mathbf{L}_2$ contains a Levi subgroup of $\mathbf{G}$. But any two Levi subgroups of $\mathbf{G}$ are conjugate of one another under an element of the unipotent radical $\mathbf{U}$ of $\mathbf{G}$ (\cite[VIII Theorem 4.3]{MR620024}). Let $u_0 \in \mathbf{U}(\mathbb{Q})$ be an element that conjugates $\mathbf{L}_1$ into a semi-simple Levi subgroup contained in $\mathbf{L}_2$. Let $u \in [G,U]$ be its projection to $G_1$ (recall that $G_1$ and $G$ are identified). Then $L_1^{u} \subset L_2$ and, by Corollary \ref{Corollary: Unipotent rational points in commensurator}, $u \in \Comm_G(\Lambda)$. 

Now,  let us explain how to deduce the case $\Rad(G) \cap L_1 \neq \{e\}$. Let $\hat{G}$ denote $G/\left(\Rad(G) \cap L_1\right)$. If $\hat{U}$ denotes the unipotent radical of $\hat{G}$, then $[\hat{G},\hat{U}]$ is the image of $[G,U]$ (e.g. Remark \ref{Remark: Alternative charac. derived subgroup}). Since $\Lambda^2 \cap \Rad(G) \cap L_1$ is an approximate lattice in $\Rad(G) \cap L_1$ (Theorem \ref{Theorem: Radical is hereditary}), the projection $\hat{\Lambda}$ of $\Lambda$ to $\hat{G}$ is an approximate lattice. We have reduced, in $\hat{G}$, to the above set-up. Finally, the element $u \in [\hat{G},\hat{U}] \cap \Comm_{\hat{G}}(\hat{\Lambda})$ can be lifted to an element of  $[G,U] \cap \Comm_{G}(\Lambda)$. Indeed, by Corollary \ref{Remark: Meyer theorem is canonical} the unipotent $\mathbb{Q}$-group $\hat{\mathbf{U}}$ given by Meyer's theorem (Proposition \ref{Proposition: Meyer's theorem $S$-adic unipotent group}) applied to $\hat{\Lambda}^2 \cap [\hat{G},\hat{U}]$ is the projection of the unipotent $\mathbb{Q}$-group $\mathbf{U}$ given by Meyer's theorem (Proposition \ref{Proposition: Meyer's theorem $S$-adic unipotent group}) applied to $\Lambda^2 \cap [G,U]$. Hence, there is $u' \in [G,U]$ in the image of $\mathbf{U}(\mathbb{Q})$ that projects to $u$ in $[\hat{G},\hat{U}]$. By Corollary \ref{Corollary: Unipotent rational points in commensurator}, $u'$ belongs to the commensurator of $\Lambda$ and $L_1^{u'} \subset L_2$. So Corollary \ref{Corollary: Choice of Levi subgroup is irrelevant} is proved.
\end{proof}

\begin{appendix}
\section{Appendix: Approximate subgroups}\label{Appendix}

In this appendix we collect a number of results pertaining to approximate subgroups in a context more general than $S$-adic linear groups. All of them are either already well-known - in which case we provide references - or simple observations.  Nevertheless, they are useful tools and can reveal particularly powerful in the right framework.  Original results contained in the last part of this appendix will also be used in the companion paper \cite{mac2023structure}.

\subsection{Some commensurability results}
We have used above a number of elementary results concerning intersections and pull-backs of approximate lattices.  For the sake of completeness we include the relevant statements below. 

 \begin{lemma}\label{Lemma: Pull-back of commensurable approximate subgroups}
  Let $\Lambda_1$ and $\Lambda_2$ be two commensurable approximate subgroups of a group $G$. Let $\phi: H\rightarrow G$ be a group homomorphism. Then $\phi^{-1}(\Lambda_1^2)$ and $\phi^{-1}(\Lambda_2^2)$ are commensurable approximate subgroups of $H$. 
 \end{lemma}
 \begin{lemma}\label{Lemma: Intersection of approximate subgroups}
Let $\Lambda_1, \ldots, \Lambda_n$ be $K_1,\ldots,K_n$-approximate subgroups of some group. We have: 
\begin{enumerate}
\item if $k_1, \ldots, k_n \geq 2$, then there is $F$ with $|F|\leq K_1^{k_1-1} \cdots K_n^{k_n-1}$ such that $$\Lambda_1^{k_1} \cap \cdots \cap \Lambda_n^{k_n} \subset F\left(\Lambda_1^2 \cap \cdots \cap \Lambda_n^2\right);$$ 
\item if $k_1, \ldots, k_n \geq 2$, then $\Lambda_1^{k_1} \cap \cdots \cap \Lambda_n^{k_n}$ is a $K_1^{2k_1-1}\cdots K_n^{2k_n-1}$-approximate subgroup.
\item if $\Lambda_1', \ldots, \Lambda_n'$ is a family of approximate subgroups such that $\Lambda_i'$ is commensurable with $\Lambda_i$ for all $1 \leq i \leq n$, then $\Lambda_1'^2 \cap \cdots \cap \Lambda_n'^2$ is commensurable with $\Lambda_1^2\cap \cdots \cap \Lambda_n^2$.
\end{enumerate}

\end{lemma}

\begin{lemma}\label{Lemma: Intersection of commensurable sets}
Take $X,Y_1,\ldots, Y_n$ subsets of a group $G$. Assume that there exist $F_1, \ldots, F_n \subset G$ finite such that $X \subset F_iY_i$ for all $i \in \{1,\dots, n\}$. Then there is $F' \subset X$ with $|F'| \leq |F_1|\cdots |F_n|$ such that $$X \subset F'\left( Y_1^{-1}Y_1 \cap \cdots \cap Y_n^{-1}Y_n\right).$$ 
\end{lemma}

 \subsection{Good models }
 The notion of good models has appeared a number of times in the considerations above.  We refer to \cite{machado2019goodmodels} for background regarding this notion.

\begin{definition*}
Let $\Lambda$ be an approximate subgroup of a group $\Gamma$ that commensurates it. A group homomorphism $f: \Gamma \rightarrow H$ with target a locally compact group $H$ is called a \emph{good model (of $(\Lambda, \Gamma)$)} if:
\begin{enumerate}
\item $f(\Lambda)$ is relatively compact;
 \item there is $U \subset H$ a neighbourhood of the identity such that $f^{-1}(U) \subset \Lambda$.
\end{enumerate}
Any approximate subgroup commensurable with an approximate subgroup that admits a good model is said \emph{laminar}.
\end{definition*}

Certain good models have particularly handy properties: 

\begin{lemma}[ \cite{machado2019goodmodels}]\label{Lemma: Bohr compactification and abstract automorphisms}
Let $\Lambda$ be an approximate subgroup of a group $\Gamma$. Suppose that $Comm_{\Gamma}(\Lambda)=\Gamma$ and that $\Lambda$ has a good model. Then there is a good model $f: \Gamma \rightarrow H_0$ of $(\Lambda, \Gamma)$ such that for any group endomorphism $a$ of $\Gamma$ such that $a(\Lambda)$ is commensurable with $\Lambda$ there is a unique continuous group endomorphism $\alpha$ of $H_0$ such that the following diagram commutes
\[\begin{tikzcd}
\Gamma \arrow{r}{f} \arrow[swap]{d}{a} & H_0 \arrow{d}{\alpha} \\
\Gamma \arrow{r}{f} & H_0
\end{tikzcd}
\]
\end{lemma}

\begin{proposition}[Prop 3.6, \cite{machado2019goodmodels}]\label{Proposition: Minimal commensurable approximate subgroup}
Let $\Lambda$ be an approximate subgroup of some group. Suppose that $\Lambda$ is laminar. Then there is an approximate subgroup $\Lambda'$ commensurable with $\Lambda$ and a good model $f: \langle \Lambda' \rangle \rightarrow H$ with target a connected Lie group and dense image. Moreover: 
\begin{enumerate}
\item if $\Lambda'' \subset \Lambda'$ is any approximate subgroup commensurable with $\Lambda'$, $\langle \Lambda'' \rangle$ has finite index in $\langle \Lambda' \rangle$;
\item we can choose $H$ without compact normal subgroup and such an $H$ is unique. 
\end{enumerate}
\end{proposition}

The relevance of good models to the study of approximate lattices is explained by: 

\begin{proposition}[\S 3.5, \cite{machado2019goodmodels}]\label{Proposition: Equivalence good models and model sets}
Let $\Lambda$ be a uniformly discrete approximate subgroup in a locally compact group $G$. Let $\Gamma \subset G$ be a subgroup of $G$ containing $\Lambda$ and commensurating it.
\begin{enumerate}
\item If $(\Lambda, \Gamma)$ has a good model $f: \Gamma \rightarrow H$, then $$\Gamma_f:=\{(\gamma, f(\gamma)) \in G \times H : \gamma \in \Gamma\}$$ 
is a discrete subgroup of $G \times H$;
\item If, moreover, $\Lambda$ is an approximate lattice and $f$ has dense image, then $\Gamma_f$ is a lattice; 
\item If $\Lambda$ is an approximate lattice, then $\Lambda$ is laminar if and only if $\Lambda$ is commensurable with a model set.
\end{enumerate}
\end{proposition}

\subsection{Approximate lattices}
We state here an elementary fact concerning approximate lattices that we have used repeatedly:
 \begin{corollary}[ Lemma A.4, \cite{hrushovski2020beyond}]\label{Corollary: Towers of star-approximate subsets are commensurable}
  If $\Lambda_1 \subset \Lambda_2$ are two approximate lattices in a locally compact second countable group $G$, then $\Lambda_1$ and $\Lambda_2$ are commensurable.
 \end{corollary}
 
 \subsection{Approximate subgroups acting on other approximate subgroups}
 
 We conclude this appendix with a simple application of Arzela--Ascoli which already has already striking consequences for approximate subgroups. 
 
\begin{lemma}\label{Lemma: Bounded approximate subgroups of of automorphisms are relatively compact}
Let $G$ be a locally compact group and $W$ be a neighbourhood of the identity. Let $\mathcal{A}$ be an approximate subgroup of the group $\Aut(G)$ of automorphisms of $G$ such that $\mathcal{A}\cdot W$ is relatively compact and $\overline{\langle \mathcal{A} \rangle}$ (in the Braconnier topology) contains the inner automorphisms. Then there is $C$ a compact normal subgroup stable under $\mathcal{A}$ such that the image of $\mathcal{A}$ in $\Aut(G/C)$ is relatively compact.
\end{lemma}

\begin{proof}
Write $\tilde{W}:=\mathcal{A}\cdot W$. For every neighbourhood of the identity $W' \subset W$ and for all finite subsets $\mathcal{F} \subset \langle \mathcal{A} \rangle$ define the neighbourhood of the identity 
$$\Omega(W', \mathcal{F}):= \bigcap_{\alpha \in \mathcal{F}\mathcal{A}} \alpha\cdot \left( \overline{\mathcal{A} \cdot W'}\right).$$  
We know that $\Omega(W', \mathcal{F})$ is compact. Moreover, if $W''^2 \subset W'$, then for any $g \in \Omega(W'', \mathcal{F})$ we have $g^2 \in \Omega(W', \mathcal{F})$. Finally, we have 
$$\left(\mathcal{F}\mathcal{A}\right)^{-1} \cdot \Omega(W', \mathcal{F}) \subset \tilde{W}.$$

Choose now a neighbourhood basis $\mathcal{N}$ at the identity and define 
$$X := \bigcap_{W' \in \mathcal{N}, \mathcal{F} \subset \langle \mathcal{A} \rangle \text{ finite}} \Omega(W', \mathcal{F}).$$  The subset $X$ is compact, and every element $g \in X$ has its conjugacy class and all its powers contained in $\tilde{W}$. According to \cite[Th. 3.11]{MR284541}, $X$ must therefore be contained in a compact subgroup $C$ invariant under all automorphisms in the closure of $\langle \mathcal{A} \rangle$ - and in particular all inner automorphisms. One readily checks by a compactness argument that the image of the subsets $\Omega(W', \mathcal{F})$ in $G/C$ must generate a neighbourhood basis. But, if $\mathcal{F}'$ is finite and such that $\mathcal{A}\mathcal{F}\mathcal{A} \subset \mathcal{F}'\mathcal{F}\mathcal{A}$, then $\mathcal{A}\cdot\Omega(W', \mathcal{F}'\mathcal{F}) \subset \Omega(W',\mathcal{F})$. So the family $\mathcal{A}$ is equicontinuous. By the Arzela--Ascoli theorem (see e.g. the proof of \cite[Prop. I.7]{MR2739075}), we deduce that $\mathcal{A}$ is relatively compact in $\Aut(G/C)$.
\end{proof}

The main purpose of Lemma \ref{Lemma: Bounded approximate subgroups of of automorphisms are relatively compact} is to upgrade some notion of normalcy. 
\begin{corollary}\label{Corollary: Bounded approximate subgroups of of automorphisms are relatively compact}
Let $\Lambda$ be an approximate subgroup generating a group $\Gamma$. Let $\Lambda_N$ be another approximate subgroup of $\Gamma$ that has a good model. Suppose that $\Lambda_N^{\Lambda}=\bigcup_{\lambda \in \Lambda}\lambda\Lambda_N\lambda^{-1}$ is commensurable with $\Lambda_N$. Then there is $\Lambda'$ commensurable with $\Lambda$ that normalises a subgroup $N$ containing $\langle \Lambda_N \rangle$ as a finite index subgroup. 
\end{corollary}

\begin{proof}
Let $c$ denote the conjugation map $\Gamma \rightarrow \Aut(\Gamma)$. For any $\gamma \in \Gamma$ we have $\gamma \Lambda_N \gamma^{-1}$ commensurable with $\Lambda_N$. Let $f: \Gamma \rightarrow H_0$ be the good model of $\Lambda_N$ from Lemma \ref{Lemma: Bohr compactification and abstract automorphisms}.  Let $a(\gamma)$ denote the element of $\Aut(H_0)$ associated with $c(\gamma)$. The uniqueness part of Lemma \ref{Lemma: Bohr compactification and abstract automorphisms} implies that $\gamma \rightarrow a(\gamma)$ is a group homomorphism. Write $\mathcal{A}:=a(\Lambda)$. Then $\mathcal{A}$ satisfies the conditions of Lemma \ref{Lemma: Bounded approximate subgroups of of automorphisms are relatively compact}. So there is a compact normal subgroup $C$ of $H_0$ stable under $\mathcal{A}$ such that the projection of $\mathcal{A}$ to $\Aut(H_0/C)$ is relatively compact. Set $U:=\overline{f(\langle \Lambda_N \rangle)}$. Then $U$ is an open subgroup of $H_0$ and $f^{-1}(U)= \langle \Lambda_N \rangle$. Using the definition of the Braconnier topology on $\Aut(H_0/C)$  we can find an approximate subgroup $\mathcal{A}'$ commensurable with $\mathcal{A}$ such that $\mathcal{A}'$ normalises $UC$. But $U$ has finite index in $UC$. So set $\Lambda'$ as the inverse image of $\mathcal{A}'^2$ through $a$ and $N$ as $f^{-1}(NC)$.  Then $N$ is normalised by $\Lambda'$. To conclude,  notice that $\Lambda'$ is commensurable with $\Lambda$ by Lemma \ref{Lemma: Pull-back of commensurable approximate subgroups}.
\end{proof}
\end{appendix}


\end{document}